\documentclass[12pt,reqno]{amsart}
\usepackage{fullpage,verbatim,paralist}
\usepackage[breaklinks,pdfstartview=FitH]{hyperref}

\usepackage{amssymb,mathrsfs,paralist}
\usepackage{graphicx}
\graphicspath{{figs/}}



\newcommand{\eqdef}{\stackrel{\mathrm{def}}{=}}
\renewcommand{\le}{\leqslant}
\renewcommand{\hat}{\widehat}
\renewcommand{\ge}{\geqslant}
\renewcommand{\leq}{\leqslant}

\renewcommand{\setminus}{\smallsetminus}
\renewcommand{\subset}{\subseteq}

\DeclareMathOperator{\lca}{lca}
  \DeclareMathOperator{\diam}{diam}

\newcommand{\N}{\mathbb{N}}
\newcommand{\E}{\mathbb{E}}

\newcommand{\Z}{\mathbb Z}
\newcommand{\p}{\mathbf{p}}
\newcommand{\F}{\mathcal{F}}
\newcommand{\G}{\mathcal{G}}
\renewcommand{\L}{\mathcal L}
\newcommand{\U}{\mathscr{U}}

\DeclareMathOperator{\depth}{depth}
\newcommand{\bd}{\overline{d}}
\newcommand{\dc}{\mathbf{c}}
\renewcommand{\P}{\mathscr P}
\newcommand{\e}{\varepsilon}
\renewcommand{\S}{\mathcal{S}}
\theoremstyle{plain}
  \newtheorem{lemma}{Lemma}[section]

  \newtheorem{theorem}[lemma]{Theorem}
  
  \newtheorem{corollary}[lemma]{Corollary}

  \theoremstyle{definition}
   \newtheorem{definition}[lemma]{Definition}
    \newtheorem{question}[lemma]{Question}
  
  \newtheorem{remark}[lemma]{Remark}

\begin{document}

\title{Ultrametric subsets with large Hausdorff dimension}\thanks{M. M. was partially supported by ISF grants 221/07 and  93/11,
BSF grants 2006009 and 2010021, and a gift from Cisco Research Center. A. N.
was partially supported by NSF grant CCF-0832795, BSF grants
2006009 and 2010021, and the Packard Foundation.
Part of this work was completed when
M.~M. was visiting Microsoft Research and University of Washington, and
A.~N. was visiting the Discrete Analysis program at the Isaac Newton Institute for Mathematical Sciences and the Quantitative Geometry program at the Mathematical Sciences Research Institute.}
\author{Manor Mendel}
\address{Mathematics and Computer Science Department, Open University of Israel, 1 University Road, P.O. Box 808
Raanana 43107, Israel} \email{mendelma@gmail.com}

\author{Assaf Naor}
\address{Courant Institute, New York University, 251 Mercer Street, New York NY 10012, USA}
\email{naor@cims.nyu.edu}

\begin{abstract}
It is shown that for every $\e\in (0,1)$, every compact metric space
$(X,d)$ has a compact subset $S\subseteq X$ that embeds into an
ultrametric space with distortion $O(1/\e)$, and $$\dim_H(S)\ge
(1-\e)\dim_H(X),$$ where $\dim_H(\cdot)$ denotes Hausdorff
dimension. 
The above $O(1/\e)$ distortion estimate is shown to be sharp via a
construction based on sequences of expander graphs.
\end{abstract}

\maketitle

\setcounter{tocdepth}{4} \tableofcontents

\section{Introduction}

Given $D\ge 1$, a metric space $(X,d_X)$ is said to embed with
distortion $D$ into a metric space $(Y,d_Y)$ if there exists $f:X\to
Y$ and $\lambda>0$ such that for all $x,y\in X$ we have
\begin{equation}\label{eq:def embedding}
\lambda d_X(x,y)\le d_Y(f(x),f(y))\le D\lambda d_X(x,y).
\end{equation}
Note that when $Y$ is a Banach space the scaling factor $\lambda$
can be dropped in the definition~\eqref{eq:def embedding}.

Answering positively a conjecture of Grothendieck~\cite{Gro53-dvo},
Dvoretzky proved~\cite{Dvo60} that for every $k\in \mathbb N$ and
$D>1$ there exists $n=n(k,D)\in \mathbb N$ such that every
$n$-dimensional normed space has a $k$-dimensional linear subspace
that embeds into Hilbert space with distortion $D$;
see~\cite{Mil71,MS99,Sch06} for the best known bounds on $n(k,D)$.

Bourgain, Figiel and Milman~\cite{BFM86} studied the following
problem as a natural nonlinear variant of Dvoretzky's theorem: given
$n\in \N$ and $D>1$, what is the largest $m\in \N$ such that {\em
any} finite metric space $(X,d)$ of cardinality $n$ has a subset
$S\subseteq X$ with $|S|\ge m$ such that the metric space $(S,d)$
embeds with distortion $D$ into Hilbert space? Denote this value of
$m$ by $R(n,D)$. Bourgain-Figiel-Milman proved~\cite{BFM86} that for
all $D>1$ there exists $c(D)\in (0,\infty)$ such that $R(n,D)\ge
c(D)\log n$, and that $R(n,1.023)=O(\log n)$. Following several
investigations~\cite{KKR94,BKRS00,BBM} that were motivated by
algorithmic applications, a more complete description of
the Bourgain-Figiel-Milman phenomenon was obtained in~\cite{BLMN05}.

\begin{theorem}[\cite{BLMN05}]\label{thm:BLMN phase}
For  $D\in (1,\infty)$ there exist $c(D),c'(D)\in (0,\infty)$ and
$\delta(D),\delta'(D)\in (0,1)$ such that for every $n\in \N$,
\begin{itemize}
\item if $D\in (1,2)$ then $c(D)\log n\le R(n,D)\le c'(D)\log
n$,
\item if $D\in (2,\infty)$ then $n^{1-\delta(D)}\le R(n,D)\le n^{1-\delta'(D)}$.
\end{itemize}
\end{theorem}
Highlighting the case of large $D$, which is most relevant for
applications, we have the following theorem.

\begin{theorem}[\cite{BLMN05,MN07,NT-fragmentations}]\label{thm:1/eps}
For every $\e\in (0,1)$ and $n\in \N$, any $n$-point metric space
$(X,d)$ has a subset $S\subseteq X$ with $|S|\ge n^{1-\e}$ that
embeds into an ultrametric space with distortion $2e/\e$. On the
other hand, there exists a universal constant $c>0$ with the
following property. For every $n\in \N$ there is an $n$-point metric
space $X_n$ such that for every $\e\in (0,1)$ all subsets
$Y\subseteq X$ with $|Y|\ge n^{1-\e}$ incur distortion at least
$c/\e$ in any embedding into Hilbert space.
\end{theorem}

Recall that a metric space $(U,\rho)$ is called an ultrametric space
if for every $x,y,z\in U$ we have $\rho(x,y)\le
\max\left\{\rho(x,z),\rho(z,y)\right\}$. Any separable ultrametric
space admits an isometric embedding into Hilbert space~\cite{VT79}.
Hence the subset $S$ from Theorem~\ref{thm:1/eps} also embeds with
the stated distortion into Hilbert space, and therefore
Theorem~\ref{thm:1/eps} fits into the Bourgain-Figiel-Milman
framework. Note, however, that the stronger statement that $S$
embeds into an ultrametric space is needed for the applications
in~\cite{BLMN05,MN07}, and that the matching lower bound in
Theorem~\ref{thm:1/eps} is for the weaker requirement of
embeddability into Hilbert space.
Thus, a byproduct of Theorem~\ref{thm:1/eps} is the assertion that, in general, the best way (up to constant factors) to find a large approximately Euclidean subset is to actually find a subset satisfying the more stringent requirement of being almost ultrametric.
The existence of the metric spaces
$\{X_n\}_{n=1}^\infty$ from Theorem~\ref{thm:1/eps} was established
in~\cite{BLMN05}. The estimate $2e/\e$ on the ultrametric distortion
of the subset $S$ from Theorem~\ref{thm:1/eps} is due
to~\cite{NT-fragmentations}, improving by a constant factor over the
bound from~\cite{MN07}, which itself improves (in an asymptotically
optimal way) on the distortion bound of
$O\left(\e^{-1}\log(2/\e)\right)$ from~\cite{BLMN05}.

In what follows, $\dim_H(X)$ denotes the Hausdorff dimension of a
metric space $X$. Inspired by the above theorems, Terence Tao
proposed (unpublished, 2006) another natural variant of the
nonlinear Dvoretzky problem:  one can keep the
statement of Dvoretzky's theorem unchanged in the context of general
metric spaces, while interpreting the notion of dimension in the
appropriate category. Thus one arrives at the following question.

\begin{question}[The nonlinear Dvoretzky problem for Hausdorff dimension]
\label{Q:tao}
Given $\alpha>0$ and $D>1$, what is the supremum over
those $\beta\ge 0$ with the following property. Every compact metric
space $X$ with $\dim_H(X)\ge \alpha$ has a subset $S\subseteq X$
with $\dim_H(S)\ge \beta$ that embeds into Hilbert space with
distortion $D$?
\end{question}
The restriction of Question~\ref{Q:tao} to compact metric spaces is
not severe. For example, if $X$ is complete and separable then one
can first pass to a compact subset of $X$ with the same Hausdorff
dimension, and even the completeness of $X$ can be replaced by
weaker assumptions; see~\cite{Car67,How95}. We will not address
this issue here and restrict our discussion to compact metric
spaces, where the crucial subtleties of the problem are already present.

Our purpose here  is to provide
answers to Question~\ref{Q:tao} in various distortion regimes, the main result being
the following theorem.

\begin{theorem}\label{thm:hausdorff into}
There exists a universal constant $C\in (0,\infty)$ such that for
every $\e\in (0,1)$ and $\alpha\in (0,\infty)$, every compact metric space $X$ with $\dim_H(X)\ge \alpha$ has a closed
subset $S\subseteq X$ with $\dim_H(S)\ge (1-\e)\alpha$ that
embeds with distortion $C/\e$ into an ultrametric space. In the
reverse direction, there is a universal constant $c>0$ such that for
every $\alpha>0$ there exists a compact metric space $X_\alpha$ with
$\dim_H(X_\alpha)= \alpha$ such that if $S\subseteq X$ satisfies
$\dim_H(S)\ge (1-\e)\alpha$ then $S$ incurs distortion at least
$c/\e$ in any embedding into Hilbert space.
\end{theorem}

The construction of the spaces
$X_\alpha$ from Theorem~\ref{thm:hausdorff into} builds  on
the examples of~\cite{BLMN05}, which are based on expander
graphs. The limiting spaces $X_\alpha$ obtained this way can therefore be
called ``expander fractals"; their construction is discussed in
Section~\ref{sec:ub}.

Our main new contribution leading to
Theorem~\ref{thm:hausdorff into} is the following structural result for general metric measure spaces. In what follows, by a metric measure space $(X,d,\mu)$ we mean a
compact metric space $(X,d)$,  equipped with a Borel measure
$\mu$ such that $\mu(X)<\infty$. For $r>0$ and $x\in X$, the corresponding closed ball is denoted $ B(x,r) = \{ y \in X:
d(x,y) \le r \}.$

\begin{theorem}\label{thm:measure}
For every $\e\in (0,1)$ there exists $c_\e\in (0,\infty)$ with the following
property. Every metric measure space $(X,d,\mu)$ has a closed subset
$S\subseteq X$ such that $(S,d)$ embeds into an ultrametric space
with distortion $9/\e$, and for every $\{x_i\}_{i\in I}\subseteq X$
and $\{r_i\}_{i\in I}\subseteq [0,\infty)$ such that the balls
$\{B(x_i,r_i)\}_{i\in I}$ cover $S$, i.e.,
\begin{equation}\label{eq:cover}
\bigcup_{i\in I} B(x_i,r_i)\supseteq S,
\end{equation}
we have
\begin{equation}\label{eq:power sum}
\sum_{i\in I} \mu(B(x_i,c_\e r_i))^{1-\e}\ge \mu(X)^{1-\e}.
\end{equation}
\end{theorem}
Theorem~\ref{thm:measure} contains Theorem~\ref{thm:1/eps} as a simple special case. Indeed, consider the case when $X$ is finite, say $|X|=n$, the measure $\mu$ is the counting measure, i.e., $\mu(A)=|A|$ for all $A\subseteq X$, and all the radii $\{r_i\}_{i\in I}$ vanish. In this case $B(x_i,r_i)=B(x_i,c_\e r_i)=\{x_i\}$, and therefore the covering condition~\eqref{eq:cover} implies that $\{x_i\}_{i\in I}\supseteq S$. Inequality~\eqref{eq:power sum} therefore implies that $|S|\ge n^{1-\e}$, which is
the (asymptotically sharp) conclusion of Theorem~\ref{thm:1/eps},
up to a constant multiplicative factor in the distortion.

Theorem~\ref{thm:measure} also implies Theorem~\ref{thm:hausdorff into}. To see this assume that $(X,d)$ is a compact metric space and $\dim_H(X)>\alpha$. The Frostman lemma (see~\cite{How95} and~\cite[Ch.~8]{Mattila}) implies that there exists a constant $K\in (0,\infty)$ and a Borel measure $\mu$ such that $\mu(X)>0$ and $\mu(B(x,r))\le Kr^\alpha$ for all $r>0$ and $x\in X$. An application of Theorem~\ref{thm:measure} to the metric measure space $(X,d,\mu)$ yields a closed subset $S\subseteq X$ that embeds into an ultrametric space with distortion $O(1/\e)$ and satisfies the covering condition~\eqref{eq:power sum}. Thus, all the covers of $S$ by balls $\{B(x_i,r_i)\}_{i\in I}$ satisfy
$$
\mu(X)^{1-\e}\le \sum_{i\in I} \mu(B(x_i,c_\e r_i))^{1-\e}\le \sum_{i\in I} \left(Kc_\e^\alpha r_i^{\alpha}\right)^{1-\e}.
$$
Hence,
$$
\sum_{i\in I} r_i^{(1-\e)\alpha}\ge \frac{\mu(X)^{1-\e}}{K^{1-\e}c_\e^{(1-\e)\alpha}}.
$$
This means that the $(1-\e)\alpha$-Hausdorff content\footnote{Recall that for $\beta\ge 0$ the $\beta$-Hausdorff content of a metric space $(Z,d)$ is defined to be the infimum of $\sum_{j\in J} r_j^\beta$ over all possible covers of $Z$ by balls $\{B(z_j,r_j)\}_{j\in J}$; see~\cite{Mattila}.} of $S$ satisfies
$$
\mathcal{H}_\infty^{(1-\e)\alpha}(S)\ge \frac{\mu(X)^{1-\e}}{K^{1-\e}c_\e^{(1-\e)\alpha}}>0,
$$
and therefore $\dim_H(S)=\inf\left\{\beta\ge 0:\ \mathcal{H}_\infty^\beta(S)=0\right\}\ge (1-\e)\alpha$, as asserted in Theorem~\ref{thm:hausdorff into}.

To summarize the above discussion, the general structural result for
metric measure spaces that is contained in Theorem~\ref{thm:measure}
implies the sharp Bourgain-Figiel-Milman style nonlinear Dvoretzky theorem
when applied to trivial covers of $S$ by singletons. The nonlinear
Dvoretzky problem for Hausdorff dimension is more subtle since one
has to argue about all possible covers of $S$, and this is achieved
by applying Theorem~\ref{thm:measure} to the metric measure space
induced by a Frostman measure.
 In both of these applications the value of the constant $c_\e$ in Theorem~\ref{thm:measure}
 is irrelevant, but we anticipate that it will play a role in future
applications of Theorem~\ref{thm:measure}. Our argument yields the
bound $c_\e=e^{O(1/\e^2)}$, but we have no reason to believe that
this dependence on $\e$ is optimal. We therefore pose the following
natural problem.
\begin{question}\label{Q:c}
What is the asymptotic behavior as $\e\to 0$ of the best possible constant $c_\e$ in Theorem~\ref{thm:measure}?
\end{question}

\subsection{An overview of the proof of
Theorem~\ref{thm:measure}}\label{sec:overview}

Theorem~\ref{thm:BLMN phase} was proved in~\cite{BLMN05} via a
deterministic iterative construction of a sufficiently large almost
ultrametric subset $S$ of a given finite metric space $(X,d)$. In
contrast, Theorem~\ref{thm:1/eps} was proved in~\cite{MN07} via a
significantly shorter probabilistic argument. It is shown
in~\cite{MN07} how to specify a distribution (depending on the
geometry of $X$) over {\em random} subsets $S\subseteq X$ that embed
into an ultrametric space with small distortion, yet their expected
cardinality is large. The lower bound on the expected cardinality of
$S$ is obtained via a lower bound on the probability $\Pr\left[x\in
S\right]$ for each $x\in X$. Such a probabilistic estimate seems to
be quite special, and we do not see how to argue probabilistically
about all possible covers of a random subset $S$, as required in
Theorem~\ref{thm:measure}. In other words, a reason why
Question~\ref{Q:tao} is more subtle than the Bourgain-Figiel-Milman
problem is that ensuring that $S$ is large is in essence a {\em
local} requirement, while ensuring that $S$ is high-dimensional is a
{\em global} requirement: once $S$ has been determined one has to
argue about all possible covers of $S$ rather than estimating
$\Pr\left[x\in S\right]$ for each $x\in X$ separately.

For the above reason our proof of Theorem~\ref{thm:measure} is a
deterministic construction which uses in some of its steps
adaptations of the methods of~\cite{BLMN05}, in addition to a
variety of new ingredients that are needed in order to handle a
covering condition such as~\eqref{eq:power sum}. Actually, in order
to obtain the sharp $O(1/\e)$ distortion bound of
Theorem~\ref{thm:measure} we also use results
of~\cite{MN07,NT-fragmentations} (see Theorem~\ref{thm:two weight}
below), so in fact Theorem~\ref{thm:measure} is based on a
combination of deterministic and probabilistic methods, the
deterministic steps being the most substantial new contribution.

The proof of Theorem~\ref{thm:measure} starts with a reduction of
the problem to the case of finite metric spaces; see
Section~\ref{sec:reduction}. Once this is achieved, the argument is
a mixture of combinatorial, analytic and geometric arguments, the
key objects of interest being {\em fragmentation maps}. These are
maps that are defined on rooted combinatorial trees and assign to
each vertex of the tree a subset of the metric space $(X,d)$ in a
way that respects the tree structure, i.e., the set corresponding to
an offspring of a vertex is a subset of the set corresponding to the
vertex itself, and vertices lying on distinct root-leaf paths are
assigned to disjoint subsets of $X$. We also require that leaves are
mapped to singletons.

Each fragmentation map corresponds to a subset of $X$ (the images of
the leaves), and our goal is to produce a fragmentation map that
corresponds to a subset of $X$ which satisfies the conclusion of
Theorem~\ref{thm:measure}. To this end, we initiate the iteration
via a bottom-up construction of a special fragmentation map; see
Section~\ref{sec:initilize}. We then proceed to iteratively ``prune"
(or ``sparsify") this initial tree so as to produce a smaller tree
whose leaves satisfy the conclusion of Theorem~\ref{thm:measure}. At
each step we argue that there must exist sufficiently many good
pruning locations so that by a pigeonhole argument we can make the
successive pruning locations align appropriately; see
Section~\ref{sec:holder}.

At this point, the subset corresponding to fragmentation map that we
constructed is sufficient to prove Theorem~\ref{thm:measure}  with a
weaker bound of $e^{O(1/\e^2)}$ on its ultrametric distortion; see
Remark~\ref{rem:simplified-argument}. To get the optimal distortion
we add another pruning step guided by a weighted version (proved in
Section~\ref{sec:NT}) of the nonlinear Dvoretzky theorem for finite
metric spaces. The mechanism of this second type of pruning is
described in Section~\ref{sec:proof-lem:3}.

It is impossible to describe the exact details of the above steps
without introducing a significant amount of notation and
terminology, and specifying rather complicated inductive hypotheses.
We therefore refer to the relevant sections for a detailed
description. To help motivate the lengthy arguments, in the body of
this paper we present the proof in a top-down fashion which is
opposite to the order in which it was described above.

\subsection{The low distortion regime}\label{sec:small}

We have thus far focused on Question~\ref{Q:tao} in the case of high
distortion embeddings into ultrametric spaces. There is also a
 Hausdorff dimensional variant  of the phase transition at distortion
$2$ that was described in Theorem~\ref{thm:BLMN phase}.

\begin{theorem}[Distortion $2+\delta$]\label{thm:2+}
There exists a universal constant $c\in (0,\infty)$ such that for
every $\delta\in (0,1/2)$, any compact metric space $(X,d)$ of finite Hausdorff dimension has a
closed subset $S\subseteq X$ that embeds with distortion $2+\delta$
in an ultrametric space, and
\begin{equation}\label{eq:delta}
\dim_H(S)\ge \frac{c\delta}{\log(1/\delta)}\dim_H(X).
\end{equation}
\end{theorem}

For distortion strictly less than $2$ the following theorem shows
that there is no  nonlinear Dvoretzky phenomenon in terms of
Hausdorff dimension.

\begin{theorem}
\label{thm:<2} For every $\alpha\in (0,\infty)$ there exists a
compact metric space $(X,d)$ of Hausdorff dimension  $\alpha$, such
that if $S\subseteq X$ embeds into Hilbert space with distortion
strictly smaller than $2$ then $\dim_H(S)=0$.
\end{theorem}

It was recently observed in~\cite{Fun11} that Theorem~\ref{thm:<2}
easily implies the following seemingly stronger assertion: there
exists a compact metric space $X_\infty$ such that
$\dim_H(X_\infty)=\infty$, yet every subset $S\subset X_\infty$ that
embeds into Hilbert space with distortion strictly smaller than~2
must have $\dim_H(S)=0$.

As in the case of the finite nonlinear Dvoretzky theorem,  Question~\ref{Q:tao} at distortion $2$ remains open. In the same vein, the correct asymptotic dependence on $\delta$ in~\eqref{eq:delta} is unknown.

Theorem~\ref{thm:2+} follows  from the following result in the
spirit of Theorem~\ref{thm:measure}, via the same Frostman measure argument.

\begin{theorem}\label{thm:measure2+} There exists a universal constant $c\in (0,\infty)$ such that
for every $\delta\in (0,1/2)$ there exists $c'_\delta\in
(0,\infty)$  with the following property. Every metric measure space
$(X,d,\mu)$ has a closed subset $S\subseteq X$ such that $(S,d)$
embeds into an ultrametric space with distortion $2+\delta$, and for
every $\{x_i\}_{i\in I}\subseteq X$ and $\{r_i\}_{i\in I}\subseteq
[0,\infty)$ such that the balls $\{B(x_i,r_i)\}_{i\in I}$ cover $S$,
we have
\begin{equation}
\sum_{i\in I}
\mu(B(x_i,c'_\delta r_i))^{\frac{c\delta}{\log(1/\delta)}}\ge
\mu(X)^{\frac{c\delta}{\log(1/\delta)}}.
\end{equation}
\end{theorem}

\subsection{Further applications}
\label{sec:app}

Several applications of our results have been recently discovered.
We discuss some
of them here as an indication of how Theorem~\ref{thm:measure} could
be used. Before doing so, we note the obvious observation that since
Theorem~\ref{thm:measure} implies Theorem~\ref{thm:1/eps} it
automatically inherits its applications in theoretical computer
science. Algorithmic applications of nonlinear Dvoretzky theory
include the best known lower bound for the randomized $k$-server
problem~\cite{BBM06,BLMN05}, and the design of a variety of
proximity data structures~\cite{MN07}, e.g., the only known
approximate distance oracles with constant query time, improving
over the important work of Thorup and Zwick~\cite{TZ05} (this
improvement is sharp~\cite{SVY09,Wul12}. Nonlinear Dvoretzky theory
is the only known method to produce such sharp constructions).

A version of Theorem~\ref{thm:hausdorff into} in the case of
infinite Hausdorff dimension, in which the conclusion is that $S$ is
also infinite dimensional, was recently obtained in~\cite{Fun11}. In
the ensuing subsections we discuss two additional research
directions: surjective cube images of spaces with large Hausdorff
dimension, and the majorizing measures theorem.

\subsubsection{Urba\'nski's problem}\label{sec:urbanski} This application of Theorem~\ref{thm:hausdorff into}
 is due to Keleti, M\'ath\'e and Zindulka~\cite{KMZ12}. We thank them for allowing us to sketch their proof
here.  Urba\'nski asked~\cite{Urb09} whether given $n\in \N$ every
metric space $(X,d)$ with $\dim_H(X)>n$ admits a surjective
Lipschitz map $f:X\to [0,1]^n$ (Urba\'nski actually needed a weaker
conclusion). Keleti, M\'ath\'e and Zindulka proved that without further
assumptions on $X$ the Urba\'nski problem has a negative answer, yet
if $X$ is an analytic subset of a Polish space then one can use
Theorem~\ref{thm:hausdorff into} to solve Urba\'nski's problem
positively. To see how this can be proved using
Theorem~\ref{thm:hausdorff into}, note that by~\cite{How95} it
suffices to prove this statement when $X$ is compact. Choose $\e\in
(0,1)$ such that $\dim_H(X)> n/(1-\e)$. By
Theorem~\ref{thm:hausdorff into} there exists a compact $S\subseteq
X$ with $\dim_H(S)>n$, an ultrametric space $(U,\rho)$, and a
bijection $f:S\to U$ satisfying $d(x,y)\le \rho(f(x),f(y))\le
\frac9{\e}d(x,y)$ for all $x,y\in S$.

Since $(U,\rho)$ is a compact ultrametric space, there exists a linear ordering $\le$ of $U$ such that for every $a,b\in U$ with $a\le b$, the order interval $[a,b]=\{c\in U:\ a\le c\le b\}$ is a Borel set satisfying $\diam([a,b])=\rho(a,b)$. This general property of ultrametric spaces follows directly from the well-known representation of such spaces as ends of trees (see~\cite{Hug04}); the desired ordering is then a lexicographical order associated to the tree structure. Since $\dim_H(U)>n$, we can consider a Frostman probability measure on $U$, i.e., a Borel measure $\mu$ on $U$ satisfying $\mu(U)=1$ such that there exists $K\in (0,\infty)$ for which $\mu(A)\le K(\diam(A))^n$ for all $A\subseteq U$. Define $g:U\to [0,1]$ by $g(a)=\mu(\{x\in U:\ x<a\})$. If $a,b\in U$ satisfy $a<b$ then $|g(b)-g(a)|=\mu([a,b))\le K(\diam([a,b]))^n=K\rho(a,b)^n$. Thus $g$ is continuous, implying that $g(U)=[0,1]$ ($U$ is compact and $g$ cannot have any ``jumps" because $\mu$ is atom-free).

Let $P:[0,1]\to [0,1]^n$ be a Peano curve (see e.g.~\cite{Sag94}), i.e., $P([0,1])=[0,1]^n$ and we have the $1/n$-H\"older estimate $\|P(s)-P(t)\|_2\le L|s-t|^{1/n}$ for all $s,t\in [0,1]$. Then the mapping $\psi=P\circ g\circ f:S\to [0,1]^n$ is surjective and $9K^{1/n}L/\e$-Lipschitz. There exists $\overline \psi:X\to [0,1]^n$ that extends $\psi$ and is $CK^{1/n}L/\e^2$-Lipschitz, where $C$ is a universal constant. This follows from the absolute extendability property of ultrametric spaces, or more generally metric trees; see~\cite{LN05}. Alternatively, one can use the nonlinear Hahn-Banach theorem~\cite[Lem.~1.1]{BL}, in which case $C$ will depend on $n$. Since $\overline{\psi}(X)\supseteq \psi(S)=[0,1]^n$, this concludes the proof of the Keleti-M\'ath\'e-Zindulka positive solution of Urba\'nski's problem.

The conclusion of Urba\'nski's problem is known to fail if we only assume that $X$ has positive $n$-dimensional Hausdorff measure; see~\cite{VIM63}, \cite{Kel95} and~\cite[Thm.~7.4]{AK00}. However, in the special case when $X$ is a subset of $\mathbb R^n$ of positive Lebesgue measure, a well-known conjecture of Laczkovich asks for the same conclusion, i.e., that there is a surjective Lipschitz mapping from $X$ onto $[0,1]^n$. The Laczkovich conjecture has a positive answer~\cite{ACP05} when $n=2$, and there is recent exciting (still unpublished) progress on the Laczkovich question for $n\ge 3$ due to Marianna Cs\"ornyei and Peter Jones. Note that the above argument implies that if $(X,d)$ is compact and $\dim_H(X)=n$ then for every $\delta\in (0,1)$ there exists a $(1-\delta)$-H\"older mapping from $X$ onto $[0,1]^n$.

\subsubsection{Talagrand's majorizing measures theorem}\label{sec:tal}  Given a metric space $(X,d)$ let $\P_X$ be the Borel probability measures on $X$. The Fernique-Talagrand $\gamma_2$ functional is defined as follows.
\begin{equation*}\label{eq:def gamma2}
\gamma_2(X,d)=\inf_{\mu\in \P_X}\sup_{x\in X} \int_0^\infty \sqrt{\log\left(\frac{1}{\mu(B(x,r))}\right)}dr.
\end{equation*}

In 1987 Talagrand proved~\cite{Tal87} the following important
nonlinear Dvoretzky-like theorem, where the notion of ``dimension"
of $(X,d)$ is interpreted to be $\gamma_2(X,d)$.
Theorem~\ref{thm:talagrand} below is stated slightly differently
in~\cite{Tal87}, but it easily follows from a combination
of~\cite[Lem.~6]{Tal87}, \cite[Thm.~11]{Tal87},
and~\cite[Prop.~13]{Tal87}.
\begin{theorem}[\cite{Tal87}]\label{thm:talagrand} There are universal constants $c,D\in (0,\infty)$ such that every finite metric space $(X,d)$ has a subset $S\subseteq X$ that embeds into an ultrametric space with distortion $D$ and $\gamma_2(S,d)\ge c\gamma_2(X,d)$.
\end{theorem}

Theorem~\ref{thm:talagrand} is of major importance since it easily
implies Talagrand's {\em majorizing measures theorem}. Specifically,
suppose that $\{G_x\}_{x\in X}$ is a centered Gaussian process and
for $x,y\in X$ we have $d(x,y)=\sqrt{\E[(G_x-G_y)^2]}$. Talagrand's
majorizing measures theorem asserts that $\E[\sup_{x\in X} G_x]\ge
K\gamma_2(X,d)$, where $K\in (0,\infty)$ is a universal constant.
Let $S\subseteq X$ be the subset obtained from an application of
Theorem~\ref{thm:talagrand} to $(X,d)$. Since ultrametric spaces are
isometric to subsets of Hilbert space, there is a Gaussian process
$\{H_x\}_{x\in S}$ such that $\rho(x,y)= \sqrt{\E[(H_x-H_y)^2]}$ is
an ultrametric on $S$ and $d(x,y)\le \rho(x,y)\le Dd(x,y)$ for all
$x,y\in S$. Hence $\gamma_2(S,\rho)\ge \gamma_2(S,d)\ge
c\gamma_2(X,d)$, and a standard application of Slepian's lemma
(see~\cite[Prop.~5]{Tal87}) yields $\E[\sup_{x\in X}
G_x]\ge\E[\sup_{x\in S} G_x]\ge D^{-1}\E[\sup_{x\in S} H_x]$. This
shows that due to Theorem~\ref{thm:talagrand} it suffices to prove
the majorizing measures theorem when $(X,d)$ is an ultrametric space
itself. Ultrametric spaces have a natural tree structure (based on
nested partitions into balls; see e.g.~\cite[Sec.~3.1]{BLMN05}), and
they can be embedded into Hilbert space so that disjoint subtrees
are orthogonal. For Gaussian processes orthogonality means
independence, which indicates how the ultrametric structure can be
harnessed to yield a direct and short proof of the majorizing
measures theorem for ultrametrics~\footnote{According to
Talagrand~\cite[Sec.~2.8]{Tal11}, Fernique was the first to observe
that the majorizing measures theorem holds for ultrametrics.
See~\cite[Prop.~13]{Tal87} for a short proof of this fact.}. This
striking application of a metric Dvoretzky-type theorem is of great
importance to several areas; we refer to~\cite{Tal05,Tal11} for an
exposition of some of its many applications.

In a forthcoming paper~\cite{MN11} we show how
Theorem~\ref{thm:measure} implies Talagrand's nonlinear Dvoretzky
theorem. The deduction of Theorem~\ref{thm:talagrand} in~\cite{MN11}
is based on the ideas presented here, but it will be published
elsewhere due to its length. Our proof of Theorem~\ref{thm:measure}
does not borrow from Talagrand's proof of
Theorem~\ref{thm:talagrand}, and we do not see how to use
Talagrand's approach in order to deduce the general covering
statement of Theorem~\ref{thm:measure}. It is an interesting open
question to determine whether Talagrand's method is relevant to the
setting of Theorem~\ref{thm:measure}. Beyond being simpler,
Talagrand's original argument has additional advantages over our
approach; specifically, it yields the important {\em generic
chaining} method~\cite{Tal05,Tal11}.

The nonlinear Dvoretzky theorems that are currently known, including
variants of Theorem~\ref{thm:talagrand} for other functionals that
are defined similarly to $\gamma_2$, differ from each other in the
notion of ``dimension", or ``largeness", of a metric space that they
use. While we now have a general nonlinear Dvoretzky theorem that
contains the Bourgain-Figiel-Milman, Talagrand and Tao phenomena as
special cases, one might conceivably obtain a characterization of
notions of ``dimension" of metric spaces for which a nonlinear
Dvoretzky theorem can be proved.  We therefore end this introduction
with an open-ended and purposefully somewhat vague direction for
future research.

\begin{question}
What are the notions of ``dimension" of metric spaces that yield a
nonlinear Dvoretzky theorem in the sense that every (compact) metric
space can be shown to contain a subset of proportional dimension
that well-embeds into an ultrametric space? At present we know this
for the following notions of dimension: $\log|X|$, $\dim_H(X)$,
$\gamma_2(X)$ (and some natural variants of these notions). Is there
an overarching principle here?
\end{question}

\section{Reduction to finite metric spaces}\label{sec:reduction}

In this section we use a simple compactness argument to show that it suffices to prove
Theorem~\ref{thm:measure} and Theorem~\ref{thm:measure2+} when
$(X,d)$ is a finite metric space. Before doing so we fix some standard terminology.

As we have already noted earlier, given a metric space $(X,d)$ and
$r>0$, the closed ball centered at $x\in X$ of radius $r$ is denoted
$B(x,r)=\{y\in X:\ d(x,y)\le r\}$. Open balls are denoted by
$B^\circ(x,r)=\{y\in X:\ d(x,y)<r\}$. We shall use this notation
whenever the metric space in question will be clear from the context
of the discussion, but when we will need to argue about several
metrics at once we will add a subscript indicating the metric with
respect to which balls are taken. Thus, we will sometimes use the
notation $B_d(x,r), B^\circ_d(x,r)$. Similar conventions hold for
diameters of subsets of $X$: given a nonempty $A\subseteq X$ we
denote $\diam(A)=\sup_{x,y\in A} d(x,y)$ whenever the underlying
metric is clear from the context, and otherwise we denote this
quantity by $\diam_d(A)$.

Given two nonempty subsets $A,B\subseteq X$ we denote as usual
\begin{equation}\label{eq:def d sets}
d(A,B)=\inf_{\substack{x\in A\\y\in B}} d(x,y),
\end{equation}
and we will also use the standard notation $d(x,A)=d(\{x\},A)$. The Hausdorff distance between $A$ and $B$ is denoted
\begin{equation}\label{eq:def haus dist}
d_H(A,B)=\max\left\{\sup_{x\in A} d(x,B),\sup_{y\in B}d(y,A)\right\}.
\end{equation}

\begin{lemma}\label{lem:compactness}
Fix $D,c\ge 1$ and $\theta\in (0,1]$. Suppose that any finite metric
measure space $(X,d,\mu)$ has a subset $S\subseteq X$ that embeds
with distortion $D$ into an ultrametric space, such that every
family of balls $\{B(x_i,r_i)\}_{i\in I}$ that covers $S$ satisfies
\begin{equation}\label{eq:compactness theta}
\sum_{i\in I} \mu\left(B(x_i,cr_i)\right)^\theta\ge \mu(X)^\theta.
\end{equation}
Then any metric measure space $(X,d,\mu)$ has a closed subset $S\subseteq
X$ that embeds with distortion $D$ into an ultrametric space, such
that every family of balls $\{B(x_i,r_i)\}_{i\in I}$ that covers $S$
satisfies~\eqref{eq:compactness theta}.
\end{lemma}

\begin{proof}
Let $(X,d,\mu)$ be a metric measure space and let $X_n$ be a
$\frac1{n}$-net in $X$, i.e., $d(x,y)>\frac1{n}$ for all distinct
$x,y\in X_n$, and $d(x,X_n)\le \frac1{n}$ for all $x\in X$. Since
$X$ is compact, $X_n$ is finite. Write
$X_n=\{x^n_1,x_2^n,\ldots,x^n_{k_n}\}$ and for $j\in
\{1,\ldots,k_n\}$ define
$$
j_n(x)\eqdef \min\left\{i\in \{1,\ldots,k_n\}:\ d\left(x,x_i^n\right)=d(x,X_n)\right\}.
$$
Consider the Voronoi tessellation $\{V_1^n,\ldots,V_{k_n}^n\}\subseteq 2^X$ given by
$$
V_j^n\eqdef \left\{x\in X:\ j_n(x)=j\right\}.
$$
Thus $\{V_1^n,\ldots,V_{k_n}^n\}$ is a Borel partition of $X$, and
we can define a measure $\mu_n$ on $X_n$ by $\mu_n(x_j^n)=\mu(V_j)$.
Note that by definition $\mu_n(X_n)=\mu(X)$.

The assumption of Lemma~\ref{lem:compactness} applied to the finite
metric measure space $(X_n,d,\mu_n)$ yields a subset $S_n\subseteq
X_n$, an ultrametric space $(U_n,\rho_n)$ and a mapping $f_n:S_n\to
U_n$ such that $d(x,y)\le \rho_n(f_n(x),f_n(y))\le Dd(x,y)$ for all
$x,y\in S_n$. Moreover, if $\{z_j\}_{j\in J}\subseteq X_n$ and
$\{r_j\}_{j\in J}\subseteq [0,\infty)$ satisfy $\bigcup_{j\in J}
B(z_j,r_j)\supseteq S_n$ then
\begin{equation}\label{eq:compactness theta-n}
\sum_{j\in J} \mu_n\left(X_n\cap B(z_j,cr_j)\right)^\theta\ge
\mu_n(X_n)^\theta=\mu(X)^\theta.
\end{equation}

Let $\U$ be a free ultrafilter over $\N$.  Since the Hausdorff
metric $d_H$ (recall~\eqref{eq:def haus dist}) on the space of closed subsets of $X$ is
compact (e.g.~\cite[Thm. 7.3.8]{Burago}), there exists a
closed subset $S\subseteq X$ such that $\lim_{n\to \U}
d_H(S,S_n)=0$.

Define $\rho:S\times S\to [0,\infty)$ as follows. For $x,y\in S$
there are $(x_n)_{n=1}^\infty,(y_n)_{n=1}^\infty\in
\prod_{n=1}^\infty S_n$ such that $\lim_{n\to \U}
d(x,x_n)=\lim_{n\to \U} d(y,y_n)=0$.  Set $\rho(x,y)=\lim_{n\to \U}
\rho_n(f_n(x_n),f_n(y_n))$. This is well defined, i.e., $\rho(x,y)$
does not depend on the choice of
$(x_n)_{n=1}^\infty,(y_n)_{n=1}^\infty$. Indeed, if
$(x_n')_{n=1}^\infty,(y_n')_{n=1}^\infty\in \prod_{n=1}^\infty S_n$
also satisfy $\lim_{n\to \U} d(x,x'_n)=\lim_{n\to \U} d(y,y'_n)=0$
then
\begin{eqnarray*}
&&\!\!\!\!\!\!\!\!\!\!\!\!\!\!\!\!\!\!\!\!\!\!\!\!\lim_{n\to \U}
\rho_n(f_n(x_n),f_n(y_n))\\&\le& \lim_{n\to \U}
\rho_n(f_n(x'_n),f_n(y'_n))+\lim_{n\to \U}
\rho_n(f_n(x_n),f_n(x'_n))+\lim_{n\to \U}
\rho_n(f_n(y_n),f_n(y'_n))\\
&\le& \lim_{n\to \U} \rho_n(f_n(x'_n),f_n(y'_n))+D\lim_{n\to \U}
d(x_n,x'_n)+D\lim_{n\to \U} d(y_n,y'_n)\\&=&\lim_{n\to \U}
\rho_n(f_n(x'_n),f_n(y'_n)),
\end{eqnarray*}
so that by symmetry $\lim_{n\to \U}
\rho_n(f_n(x_n),f_n(y_n))=\lim_{n\to \U}
\rho_n(f_n(x'_n),f_n(y'_n))$. It is immediate to check that
$d(x,y)\le \rho(x,y)\le Dd(x,y)$ for all $x,y\in S$ and that $\rho$
is an ultrametric on $S$.

Now, let $\{x_i\}_{i=1}^\infty\subseteq X$ and $\{r_i\}_{i=1}^\infty\subseteq [0,\infty)$ satisfy $\bigcup_{i=1}^\infty
B(x_i,r_i)\supseteq S$. Fix $\eta>0$. For every $i\in \N$ there is
$\e_i\in (0,1)$ such that
\begin{equation}\label{eq:perturbation}
\mu\left(B(x_i, cr_i+c\e_i)\right)\le
\mu\left(B(x_i,cr_i)\right)+\left(\frac{\eta}{2^i}\right)^{1/\theta}.
\end{equation}
Since $S$ is compact, there exists a finite subset $I_\eta\subseteq
\N$ such that $\bigcup_{i\in I_\eta} B(x_i,r_i+\e_i/2)\supseteq S$.
Denote $\e=\min_{i\in I_\eta} \e_i$. By definition of $S$ there
exists $n\in \N$ such that $n>8/\e$ and $d_H(S_n,S)<\e/8$. For every
$i\in I_\eta$ let $z_i^n\in X_n$ satisfy $d(z_i^n,x_i)=d(x_i,X_n)\le
1/n<\e/8$. Now, $\bigcup_{i\in I_\eta} B(z_i^n,r_i+3\e_i/4)\supseteq
S_n$ because $d_H(S_n,S)<\e/8$ and $\bigcup_{i\in I_\eta}
B(x_i,r_i+\e_i/2)\supseteq S$. An application
of~\eqref{eq:compactness theta-n} now yields the bound
\begin{equation}\label{eq:use theta}
\sum_{i\in I_\eta}\mu_n\left(X_n\cap
B\left(z_i^n,cr_i+\frac{3c\e_i}{4}\right)\right)^\theta\ge \mu(X)^\theta.
\end{equation}
By the definition of $\mu_n$,
\begin{eqnarray}\label{eq:pass to cont}
&&\!\!\!\!\!\!\!\!\!\!\!\!\!\!\!\!\!\!\!\!\!\!\!\!\!\!\!\!\!\!\nonumber\mu_n\left(X_n\cap
B\left(z_i^n,cr_i+\frac{3c\e_i}{4}\right)\right)\\&=&\mu\left(\bigcup\left\{V_j^n:\
j\in \{1,\ldots,k_n\}\ \wedge\  x_j^n\in
B\left(z_i^n,cr_i+\frac{3c\e_i}{4}\right)\right\}\right)\nonumber\\
&\le&\mu\left(B\left(z_i^n,cr_i+\frac{3c\e_i}{4}+\frac{1}{n}\right)\right)\nonumber\\
&\le& \mu\left(B\left(x_i,cr_i+\frac{3c\e_i}{4}+\frac{2}{n}\right)\right) \nonumber\\
&\le& \mu\left(B\left(x_i,cr_i+c\e_i\right)\right)\nonumber\\
&\stackrel{\eqref{eq:perturbation}}{\le}& \mu\left(B(x_i,cr_i)\right)+\left(\frac{\eta}{2^i}\right)^{1/\theta}.
\end{eqnarray}
Hence,
\begin{multline}\label{eq:eta->0}
\mu(X)^\theta\stackrel{\eqref{eq:use theta}\wedge\eqref{eq:pass to cont}}{\le}\sum_{i\in I_\eta}\left(\mu\left(B(x_i,cr_i)\right)+\left(\frac{\eta}{2^i}\right)^{1/\theta}\right)^\theta\\\le
\sum_{i\in I_\eta}\left(\mu\left(B(x_i,cr_i)\right)^\theta+\frac{\eta}{2^i}\right)\le \eta+\sum_{i=1}^\infty\mu\left(B(x_i,cr_i)\right)^\theta.
\end{multline}
Since~\eqref{eq:eta->0} holds for all $\eta>0$, the proof of Lemma~\ref{lem:compactness} is complete.
\end{proof}

\medskip
\noindent{\bf Assumptions.}
 Due to  Lemma~\ref{lem:compactness} we assume from here through the end
of Section~\ref{sec:NT}  that
$(X,d,\mu)$ is a finite metric measure space. By restricting to the support of $\mu$, we assume throughout that $\mu(\{x\})>0$ for all $x\in X$. By rescaling the
metric, assume also that $\diam(X)=1$.

\section{Combinatorial trees and fragmentation maps}
\label{sec:comb}
The ensuing arguments rely on a variety of constructions involving combinatorial trees. We will work only with finite rooted trees, i.e., finite graph-theoretical trees $T$ with a distinguished vertex $r(T)$ called the root of $T$. We will slightly abuse notation by identifying $T$ with its vertex set, i.e., when we write $v\in T$ we mean that $v$ is a vertex of $T$. We shall say that $u\in T$ is an ancestor of $v\in T\setminus \{u\}$ if $u$ lies on the path joining $v$ and $r(T)$. In this case we also say that $v$ is a descendant of $u$. We say that $v$ is a weak descendant (respectively weak ancestor) of $u$ if it is either a descendant (respectively ancestor) of $u$ or $v=u$. If $u$ is either a weak ancestor of $v$ or a weak descendant of $v$ we say that $u$ and $v$ are comparable, and otherwise we say that they are incomparable. The leaves of $T$, denoted $\L(T)\subseteq T$, is the set of vertices of $T$ that do not have descendants.

\begin{definition}[Cut set]\label{def:cut set}
Let $T$ be a rooted tree. A subset $S\subseteq T$ is called a cut set of $T$ if any root-leaf path in $T$ intersects $S$. Equivalently, $S$ is a cut set of $T$ if every $u\in T$ is comparable to a vertex in $S$. See~\cite[Ch.~4 \& Sec.~12.4]{MP10}.
\end{definition}

If $v\in T\setminus \{r(T)\}$ then we denote by $\p(v)=\p_T(v)$ its parent in $T$, i.e., the vertex adjacent to $v$ on the path joining $v$ and $r(T)$. We say that $v\in T\setminus\{r(t)\}$ is a child of $u\in T$ if $\p(v)=u$, and the set $\p^{-1}(u)=\{v\in T:\ \p(v)=u\}$ is the set of children of $u$. Thus $\L(T)=\{u\in T:\ \p^{-1}(u)=\emptyset\}$. If $u,v\in T\setminus\{r(T)\}$ are distinct and satisfy $\p(u)=\p(v)$ then we say that $u$ and $v$ are siblings in $T$.

The depth of $u\in T$, denoted $\depth_T(u)$, is  the number of edges on the path joining $u$ and $r(T)$. Thus $\depth_T(r(T))=0$. The least common ancestor of $u,v\in T$, denoted $\lca(u,v)=\lca_T(u,v)$, is the vertex of maximal depth that is an ancestor of both $u$ and $v$.

\begin{definition}[Subtree]
\label{def:subtree} Let $T$ be a finite rooted tree. A subtree $T'$
of $T$ is a connected rooted subgraph of $T$ whose set of leaves is
a subset of the leaves of $T$, i.e., $\L(T')\subseteq \L(T)$.
\end{definition}

Given $u\in T$, we denote by $T_u\subseteq T$ the subtree rooted at
$u$, i.e., the tree consisting of all the weak descendants of $u$ in
$T$, with the edges inherited from $T$. Thus $r(T_u)=u$.

\begin{definition}\label{def:D*}
Let $T$ be a rooted tree and $A\subseteq T$. For $u\in T$ define
$D_T(u,A)\subseteq T$ to be the set of all $v\in A$ such that
$v$ is a descendant of $u$ and no ancestor of $v$ is also in $A$ and
is a descendant of $u$. Note that $D_T(u,A)=\emptyset$ if $u$ has no descendants in $A$, and  $D_T(u,A)$ is a cut
set of the subtree $T_u$ if $A\cap T_u$ is a cut-set in $T_u$ (this happens in particular if $A$ contains the leaves of $T_u$). We also define
\begin{equation}\label{eq:def D^*}
D_T^*(u,A)=\left\{\begin{array}{ll}D_T(u,A) & \mathrm{if}\ u\in T\setminus A,\\
\{u\}& \mathrm{if}\ u\in A.\end{array}\right.
\end{equation}
\end{definition}

\medskip

 Trees interact with metric spaces via the notion of fragmentation maps.
\begin{definition}[Fragmentation map]\label{def:frag map}
Let $(X,d)$ be a finite metric space. A fragmentation map of $X$ is a function $\F:T\to 2^X$, where $T$ is a finite rooted tree that satisfies the following conditions.
\begin{itemize}
\item $\F(r(T))=X$.
\item If $v\in \L(T)$ is a leaf of $T$ then $\F(v)$ is a singleton, i.e., $\F(v)=\{x\}$ for some $x\in X$.
\item If $v\in T\setminus \{r(T)\}$ the $\F(v)\subseteq \F(\p(v))$.
\item If $u,v\in T$ are incomparable then $\F(u)\cap\F(v)=\emptyset$.
\end{itemize}
In what follows, given a fragmentation map $\F:T\to 2^X$, we will use the notation $\F_u=\F(u)$.
\end{definition}

\begin{definition}[Boundary of a fragmentation map]\label{def:boundary}
The boundary of a fragmentation map $\F:T\to 2^X$ is a new map $\partial \F: T\to 2^X$ defined as follows. For $u\in T$ the set $\partial \F(u)=\partial \F_u$ is the subset of $X$ corresponding to the image under $\F$ of the leaves of the subtree $T_u$, i.e.,
$$
\partial \F_u =\bigcup_{v\in \L(T_u)}\F_v.
$$
Note that we always have $\partial \F_u\subseteq \F_u$.
\end{definition}

\begin{definition}[Partition map]\label{def:partition map}
A partition map is a fragmentation map $\F:T\to 2^X$ such that $\partial \F_{r(T)}=X$.
Note that in this case $\partial \F=\F$.
\end{definition}

Up to this point the metric on $X$ did not play any role.
The following definition is one out of  two definitions that tie  the structure of a fragmentation map $\F:T\to 2^X$ to the geometry of $X$
(the second definition, called \emph{the separation property}, will be introduced in Section~\ref{sec:proof-lem:3}).

\begin{definition}[Lacunary fragmentation map]\label{def:lacunary}
Given $K,\gamma\in (0,\infty)$, a fragmentation map $\F:T\to 2^X$ is $(K,\gamma)$-lacunary if for every $q\in T$ and every $u\in T$ such that $u$ is a weak descendant of $q$ and $u$ has at least two children, i.e., $|\p^{-1}(u)|>1$, we have
\begin{equation}\label{eq:def lacunary}
\diam\left(\F_q\right)\le K\gamma^{\depth_T(u)-\depth_T(q)}\cdot\min_{\substack{v,w\in \p^{-1}(u)\\v\neq w}} d\left(\partial\F_v,\partial\F_w\right).
\end{equation}
\end{definition}

\begin{lemma}\label{lem:sep->ultra}
Let $\F:T\to X$ be a $(K,\gamma)$-lacunary fragmentation map of a metric space $(X,d)$. Then $(\partial\F_{r(T)},d)$ embeds with distortion $K$ into an ultrametric space.
\end{lemma}

\begin{proof}
For $x,y\in \partial\F_{r(T)}$ there are $a,b\in \L(T)$ such that $\F_a=\{x\}$ and $\F_b=\{y\}$. Define
$$
\rho(x,y)=\diam\left(\F_{\lca(a,b)}\right).
$$
Since $x,y\in \F_{\lca(a,b)}$ we have $d(x,y)\le \rho(x,y)$. Assume that $x\neq y$. Then $\lca(a,b)$ has distinct children $v,w\in \p^{-1}(\lca(a,b))$ such that $x\in \partial\F_v$ and $y\in \partial\F_w$. An application of~\eqref{eq:def lacunary} to $q=u=\lca(a,b)$ shows that $\rho(x,y)\le Kd(x,y)$. It remains to note that $\rho$ is an ultrametric. Indeed, take $a,b,c\in \L(T)$ and write $\F_a=\{x\}$, $\F_b=\{y\}$, $\F_c=\{w\}$. If $\lca(a,b)$ is a weak descendant of $\lca(b,c)$ then $\F_{\lca(a,b)}\subseteq \F_{\lca(b,c)}$, implying that $\rho(x,y)\le \rho(y,z)$. Otherwise $\lca(a,b)\in \{\lca(a,c),\lca(b,c)\}$, implying that $\rho(x,y)=\max\{\rho(x,z),\rho(y,z)\}$.
\end{proof}

The proof of Lemma~\ref{lem:sep->ultra} did not use the full
strength of Definition~\ref{def:lacunary}. Specifically, the
parameter $\gamma$ did not appear, and we could have used a weaker
variant of~\eqref{eq:def lacunary} in which the left hand side is
$\diam(\partial \mathcal F_q)$ instead of $\diam(\mathcal F_q)$. The
full strength of the $(K,\gamma)$-lacunary condition will be used in
the ensuing arguments since they allow us to have better control on
restrictions of fragmentation maps to subtrees of $T$.

\section{From fragmentation maps to covering theorems}

Here we show how a lacunary fragmentation map which satisfies a certain cut-set inequality can be used to prove a covering theorem in the spirit of the conclusion of Theorem~\ref{thm:measure}. This is the content of the following lemma.

\begin{lemma} \label{lem:cut-set-to-cover}
Fix $K,\gamma\in (0,\infty)$ and $\theta\in (0,1)$.
Let $(X,d,\mu)$ be a finite metric measure space. Assume that there exists a $(K,\gamma)$-lacunary fragmentation map $\G:T\to 2^X$ such that every leaf $\ell\in\L(T)$ has no siblings, and furthermore $\G_{\p(\ell)}=\G_{\ell}$. Suppose also that for any cut-set $G$ of $T$ we have
\[ \sum_{v\in G} \mu(\G_{\p(v)})^\theta \ge \mu(X)^\theta,\]
where  if $v\in T$ is the root then we set $\p(v)=v$.

Then
for any $\{x_i\}_{i\in I}\subseteq X$ and $\{r_i\}_{i\in I}\subseteq [0,\infty)$ such that the $d$-balls $\{B_d(x_i,r_i)\}_{i\in I}$ cover $\partial\G_{r(T)}$, we have
\begin{equation}\label{eq:expanded ball}
 \sum_{i\in I} \mu\left(B_d\left(x_i,\left(1+2K^2\gamma\right)r_i\right)\right)^{\theta} \ge  \mu(X)^{\theta} .
 \end{equation}
\end{lemma}
\begin{proof}
Without loss of generality assume that $\partial\G_{r(T)}\cap B_d(x_i,r_i)\ne \emptyset$ for all $i\in I$.
Let $\rho$ be the ultrametric induced by $\G$ on $\partial\G_{r(T)}$, as constructed in the proof of Lemma~\ref{lem:sep->ultra}. Thus for $x,y\in \partial \G_{r(T)}$ we have $\rho(x,y)=\diam_d\left(\G_{\lca(a,b)}\right)$, where $a,b\in \L(T)$ satisfy $\G_a=\{x\}$ and $\G_{b}=\{y\}$.
Note that
this definition implies that
\begin{equation}\label{eq:diam rho}
\forall v\in T,\quad \diam_\rho(\partial \G_v)=\diam_d(\G_v).
\end{equation}
By Lemma~\ref{lem:sep->ultra} we know that
\begin{equation}\label{eq:K-equiv}
\forall x,y\in \partial\G_{r(T)},\quad d(x,y) \le \rho(x,y)\le K d(x,y).
\end{equation}

For every $i\in I$ choose $y_i\in \partial\G_{r(T)}$ satisfying
\begin{equation}\label{eq:closest def}
d(x_i,y_i)=\min_{y\in \partial\G_{r(T)}} d(x_i,y)\le r_i.
\end{equation}
Then $B_d(y_i,2r_i)\supseteq B_d(x_i,r_i)$.
Hence the balls $\{B_d(y_i,2r_i)\}_{i\in I}$ also cover $\partial\G_{r(T)}$.  By~\eqref{eq:K-equiv} we have
$B_\rho(y_i,2K r_i) \supseteq  B_d(y_i,2r_i)\cap \partial\G_{r(T)}$, so we also know that the $\rho$-balls $\{B_\rho(y_i,2K r_i)\}_{i\in I}$ cover $\partial\G_{r(T)}$.

For $i\in I$ choose $v_i\in T$ as follows.
If
$B_\rho(y_i,2K r_i)$ is a singleton then $v_i$ is defined to be the leaf of $T$ such that $\G_{v_i}=\{y_i\}$.
Otherwise pick $v_i$ to be the highest ancestor of  $y_i$ in $T$ such that
\begin{equation}\label{eq:diam bound v_i}
\diam_\rho(\partial \G_{v_i}) \le 2K r_i
\end{equation}
and $v_i$ has at least two children. Then
 $B_\rho(y_i,2K r_i)=\partial\G_{v_i}$. Hence, since $\{B_\rho(y_i,2K r_i)\}_{i\in I}$ cover $\partial\G_{r(T)}$, we have  $\bigcup_{i\in I} \L \left(T_{v_i}\right) = \L(T)$. If we set $G=\{v_i\}_{i\in I}$ then we conclude that $G$ is a cut-set of $T$. Our assumption therefore implies that
\begin{equation}\label{eq:use parent lemma}
 \sum_{i\in I} \mu\left(\G_{\p(v_i)}\right)^{\theta} \ge
 \mu(X)^{\theta} .
 \end{equation}

When $v_i$ has at least two children we deduce from the fact that $\G$ is $(K,\gamma)$-lacunary that
\begin{equation}\label{eq:use lacunary v_i}
\diam_d\left(\G_{\p(v_i)}\right)\le K\gamma \diam_d(\partial\G_{v_i}).
\end{equation}
(Recall Definition~\ref{def:lacunary} with $q=\p(v_i)$ and $u=v_i$.) When $v_i$ is a leaf
our assumptions imply that $\G_{v_i}$ and $\G_{\p(v_i)}$ are both singletons, and therefore  $\diam_d\left(\G_{\p(v_i)}\right)=\diam_d\left(\partial\G_{v_i}\right)=0$, so~\eqref{eq:use lacunary v_i} holds in this case as well. Hence for every $i\in I$ and $z\in \G_{\p(v_i)}$ we have
\begin{multline}\label{eq:punchline cover}
d(z,x_i)\le d(x_i,y_i)+d(z,y_i)\stackrel{\eqref{eq:closest def}}{\le} r_i+d(z,y_i)\stackrel{(\clubsuit)}{\le} r_i+\diam_d\left(\G_{\p(v_i)}\right)\\\stackrel{\eqref{eq:use lacunary v_i}}{\le} r_i+K\gamma \diam_d(\partial\G_{v_i})
\stackrel{\eqref{eq:K-equiv}}{\le}
r_i+K\gamma \diam_\rho(\partial\G_{v_i})\stackrel{\eqref{eq:diam bound v_i}}{\le} r_i+2K^2\gamma r_i,
\end{multline}
where $(\clubsuit)$ follows from the fact that $y_i\in \partial\G_{v_i}\subseteq \G_{v_i} \subseteq \G_{\p(v_i)}$. The validity of~\eqref{eq:punchline cover} for all $z\in \G_{\p(v_i)}$ is the same as the inclusion $\G_{\p(v_i)}\subseteq B_d\left(x_i,\left(1+2K^2\gamma\right)r_i\right)$. Now~\eqref{eq:expanded ball} follows from~\eqref{eq:use parent lemma}.
\end{proof}

In light of Lemma~\ref{lem:cut-set-to-cover}, our goal is to construct a fragmentation map $\F:T\to 2^X$ satisfying the assumptions of
Lemma~\ref{lem:cut-set-to-cover} with $\theta=1-\e$, such that $(\partial F_{r(T)},d)$ embeds into an ultrametric space with distortion $O(1/\e)$.  Note that the $(K,\gamma)$-lacunary assumption in Lemma~\ref{lem:cut-set-to-cover}
implies by Lemma~\ref{lem:sep->ultra} that $(\partial F_{r(T)},d)$ embeds into an ultrametric  space with distortion $K$. However, more work will be needed in order to obtain the desired $O(1/\e)$ distortion.

In what follows we use the following notation.

\begin{definition}\label{def:theta(D)}
Given $D\in (2,\infty)$ let $\theta(D)\in (0,1)$ denote the unique solution of the equation \begin{equation}\label{eq:def theta(D)}
\frac{2}{D}=(1-\theta)\theta^{\frac{\theta}{1-\theta}}.\end{equation}
It is elementary to check that
\begin{equation}\label{eq:lower theta 2e}
\forall D\in (2,\infty),\quad\theta(D)\ge 1-\frac{2e}{D},\end{equation} and
\begin{equation}\label{eq:lower theta near 2}\forall\delta\in (0,1/2),\quad\theta(2+\delta)\ge \frac{c\delta}{\log(1/\delta)},\end{equation}
 where $c\in (0,\infty)$ is a universal constant.
\end{definition}

The following key lemma describes the fragmentation map that we will construct.

\begin{lemma} \label{lem:3} Fix $D\in (2,\infty)$, an integer $k\ge 2$, and $\tau\in\left(0,\frac{D-2}{3D+2}\right)$. Let $(X,d,\mu)$ be  a finite metric measure space of diameter $1$. Then there exists a fragmentation map $\G:T\to 2^X$ with the following properties.
\begin{itemize}
\item Every leaf $\ell\in\L(T)$ has no siblings, and furthermore $\G_{\p_T(\ell)}=\G_{\ell}$.
\item
$\G$ is
$\left(\frac{2}{1-3\tau} \tau^{-4k^2},\tau^{-4k^2}\right)$-lacunary.
\item $\left(\partial\G_{r(T)},d\right)$ embeds into an ultrametric space with distortion $D$.
\item Every cut-set $G\subseteq T$ (recall Definition~\ref{def:cut set}) satisfies
\begin{equation}\label{eq:lemma goal}
\sum_{v\in G}\mu\left(\p_T(v)\right)^{\left(1-\frac1{k}\right)^2\theta\left(\frac{1-3\tau}{1+\tau}D\right)}\ge \mu(X)^{\left(1-\frac1{k}\right)^2\theta\left(\frac{1-3\tau}{1+\tau}D\right)},
\end{equation}
where $\p_T(v)$ is the parent of $v$ in $T$ if $v$ is not the root, and the root if $v$ is the root.
\end{itemize}
\end{lemma}

Lemma~\ref{lem:3} will be proved in Section~\ref{sec:proof-lem:3}. Assuming its validity for the moment, we now proceed to use it in combination with Lemma~\ref{lem:cut-set-to-cover} to prove Theorem~\ref{thm:measure}
 and Theorem~\ref{thm:measure2+}.

\begin{proof}[Proof of Theorem~\ref{thm:measure}]
By Lemma~\ref{lem:compactness} we may assume that $(X,d,\mu)$ is a finite metric measure space. Fix an integer $ \frac{10}{\e}\le k\le \frac{11}{\e}$ and set $\tau=\frac{1}{20}$. Then $(1-\e)\left(1-\frac1{k}\right)^{-2}\in (0,1)$, so we can define
\begin{equation}\label{eq:def D punchline}
D=\frac{1+\tau}{1-3\tau}\cdot \theta^{-1}\left(\frac{1-\e}{\left(1-\frac{1}{k}\right)^2}\right)=\frac{11}{7}\theta^{-1}\left(\frac{1-\e}{\left(1-\frac{1}{k}\right)^2}\right).
\end{equation}
Equivalently,
$$
\left(1-\frac1{k}\right)^2\theta\left(\frac{1-3\tau}{1+\tau}D\right)=1-\e.
$$

Due to~\eqref{eq:def theta(D)}, for every $s\in (0,1)$ we have $\theta^{-1}(s)=2(1-s)^{-1}s^{-s/(1-s)}>2$. Hence it follows from~\eqref{eq:def D punchline} that $D>2(1+\tau)/(1-3\tau)$, or equivalently $\tau<(D-2)/(3D+2)$. By~\eqref{eq:lower theta 2e} we have $\theta^{-1}(s)\le 2e/(1-s)$. Therefore,
\begin{equation}\label{eq:9}
D\le \frac{11}{7}\cdot\frac{2e}{1-\frac{1-\e}{(1-\e/10)^2}}=\frac{42e(10-\e)^2}{17\e(8+\e)}\le \frac{9}{\e},
\end{equation}
where the last inequality in~\eqref{eq:9} is elementary. The required conclusion now follows from Lemma~\ref{lem:3} and Lemma~\ref{lem:cut-set-to-cover}. Note that we get the bound $c_\e=\tau^{-O(k^2)}=e^{O(1/\e^2)}$.
\end{proof}

\begin{proof}[Proof of Theorem~\ref{thm:measure2+}] Again, using Lemma~\ref{lem:compactness} we may assume that $(X,d,\mu)$ is a finite metric measure space.
Apply Lemma~\ref{lem:3} with $D=2+\delta$, $k=2$ and $\tau=\delta/9$. Denote the exponent in~\eqref{eq:lemma goal} by $s=\frac12\theta\left((9-3\delta)(2+\delta)/(9+\delta)\right)$. By~\eqref{eq:lower theta near 2} there is a universal constant $c\in (0,\infty)$ such that $s\ge t$, where $t=c\delta/\log(1/\delta)$. Let $\G:T\to 2^X$ be the fragmentation obtained obtained from Lemma~\ref{lem:3}, and let $G$ be a cut-set in $T$. Then by~\eqref{eq:lemma goal} we have,
$$
\left(\sum_{v\in G}\mu\left(\p_T(v)\right)^t\right)^{1/t}\ge \left(\sum_{v\in G}\mu\left(\p_T(v)\right)^s\right)^{1/s}\ge \mu(X).
$$
We can therefore apply Lemma~\ref{lem:cut-set-to-cover} with $\theta=t$, $K=2\tau^{-16}/(1-3\tau)$ and $\gamma=\tau^{-16}$, obtaining Theorem~\ref{thm:measure2+}. Note that this shows that $c'_\delta$ can be taken to be a constant multiple of $\delta^{-16}$.
\end{proof}

\section{Asymptotically optimal fragmentation maps:  proof of Lemma~\ref{lem:3}}
\label{sec:proof-lem:3}

It remains to prove Lemma~\ref{lem:3} in order to establish
Theorem~\ref{thm:measure} and Theorem~\ref{thm:measure2+}. The proof
of Lemma~\ref{lem:3} decomposes naturally into two parts. The first
part yields a fragmentation map $\F:T\to 2^X$ that satisfies the
desired cut-set inequality~\eqref{eq:lemma goal}, but the distortion
of $(\partial \F_{r(T)},d)$ in an ultrametric space is not good
enough. The second part improves the embeddability of $(\partial
\F_{r(T)},d)$ into an ultrametric space by performing further
pruning.

We begin with the second part since it is shorter and simpler to
describe. In order to be able to improve the embeddability of
$(\partial \F_{r(T)},d)$ into an ultrametric space, we will use the
following property.

\begin{definition}[Separated fragmentation map]\label{def:separated}
Given $\beta>0$ and a fragmentation map $\F:T\to 2^X$, a vertex $u\in T$ is called $\beta$-separated if for every $x\in \left(\partial \F_{r(T)}\right)\setminus \left(\partial \F_u \right)$ we have
\begin{equation}\label{eq:def separated}
d\left(x,\F_u\right)\ge \beta\cdot \diam\left(\F_u\right).
\end{equation}
The map $\F:T\to 2^X$ is called $\beta$-separated if all the vertices $u\in T$ are $\beta$-separated.
\end{definition}

The following very simple lemma exploits the fact that the class of
ultrametrics is closed under truncation. This fact will serve as a
useful normalization in the ensuing arguments.

\begin{lemma}\label{lem:truncation}
Let $(X,d)$ be a bounded metric space that embeds with distortion $D$ into an ultrametric space. Then there exists an ultrametric $\rho$ on $X$ satisfying
\begin{itemize}
\item  $d(x,y)\le \rho(x,y)\le Dd(x,y)$ for all $x,y\in X$,
\item $\diam_d(X)=\diam_\rho(X)$.
\end{itemize}
\end{lemma}
\begin{proof}
We are assuming that there exist $A,B>0$ and an ultrametric $\rho_0$ on $X$ that satisfies $Ad(x,y)\le \rho_0(x,y)\le Bd(x,y)$ for all $x,y\in X$, where $B/A\le D$. We can therefore define $\rho=\min\{\rho_0/A,\diam_d(X)\}$.
\end{proof}

The last ingredient that we need before we can state and prove the
lemma that describes how to improve the embeddability of $(\partial
\F_{r(T)},d)$ into an ultrametric space  is a weighted version of
nonlinear Dvoretzky theorem for finite metric spaces. As discussed
in the introduction, it was proved in~\cite{NT-fragmentations} that
for every $D>1$, every $n$-point metric space $(X,d)$ contains a
subset of size $n^{\theta(D)}$ that embed in ultrametric with
distortion at most $D$, where $\theta(D)$ is defined
in~\eqref{eq:def theta(D)}. We will need  the following
generalization of this result.

\begin{theorem} \label{lem:theta(D)}
For every $D>2$, every finite metric space $(X,d)$ and every $w:X\to (0,\infty)$, there exists $S\subseteq X$ that embeds with distortion $D$ into an ultrametric space and satisfies,
\begin{equation}\label{eq:theta(D) condition}
\sum_{x\in S} w(x)^{\theta(D)}\ge \left(\sum_{x\in X} w(x)\right)^{\theta(D)}.
\end{equation}
\end{theorem}
With some minor changes, the proof in~\cite{NT-fragmentations} also
applies to the more general weighted setting of
Theorem~\ref{lem:theta(D)}. We prove Theorem~\ref{lem:theta(D)}  in
Section~\ref{sec:NT} by sketching the necessary changes to the
argument in~\cite{NT-fragmentations}.

Assuming the validity of Theorem~\ref{lem:theta(D)}, we are now ready
to improve the ultrametric distortion of a fragmentation map by
performing additional pruning. We use the ``metric composition
technique" of~\cite{BLMN05}, which takes a vertex and its children
in the tree associated to the fragmentation map, deletes some of
these children, and arranges the remaining children into a new tree
structure. The deletion is done by solving a nonlinear Dvoretzky
problem for weighted finite metric spaces, i.e., by applying
Theorem~\ref{lem:theta(D)}.

\begin{lemma} \label{lem:2.5} Fix $D\in (2,\infty)$ and $\beta\in (0,\infty)$.
Let $(X,d)$ be a finite metric space. Suppose that $\F:T\to 2^X$ is a fragmentation map which is $\beta$-separated.
 Suppose also that there is a weight function $w:T\to (0,\infty)$ which is subadditive, i.e., that for every non-leaf vertex $u\in T\setminus \L(T)$,
\begin{equation}\label{eq:sub in compose lem}
 \sum_{v\in\p_T^{-1}(u)} w(v) \ge w(u),
 \end{equation}
Then there exists a subtree $T'$ of $T$ with the same root such that the restricted fragmentation map $\G=\F|_{T'}$ satisfies the following properties.
\begin{itemize}
\item $\left(\partial\G_{r(T')},d\right)$ embeds into an ultrametric space with distortion $D\left(1+\frac{2}{\beta}\right)$.
\item Every non-leaf vertex $u\in T'\setminus \L(T')$ satisfies
\begin{equation}\label{eq:weight-2.5}
\sum_{v\in\p_{T'}^{-1}(u)} w(v)^{\theta(D)} \ge w(u)^{\theta(D)}.
\end{equation}
\end{itemize}
\end{lemma}

\begin{proof}
Before delving into the details of proof, the reader may want to
consult Figure~\ref{fig:separation}, in which the strategy of the
proof is illustrated.

\begin{figure}[ht]
\centering \fbox{
\begin{minipage}{6in}
\centering \includegraphics[scale=0.3,page=3]{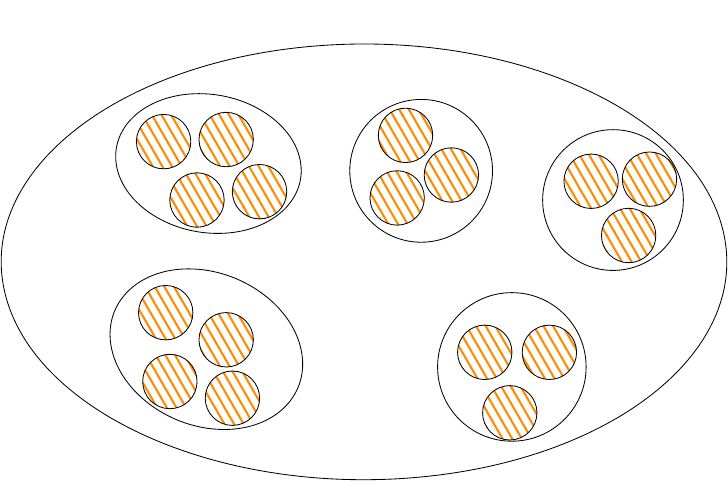}
\includegraphics[scale=0.3,page=4]{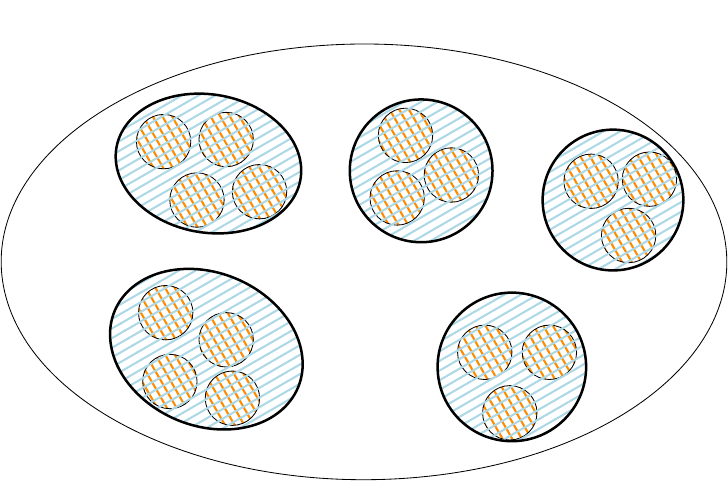}
\includegraphics[scale=0.3,page=5]{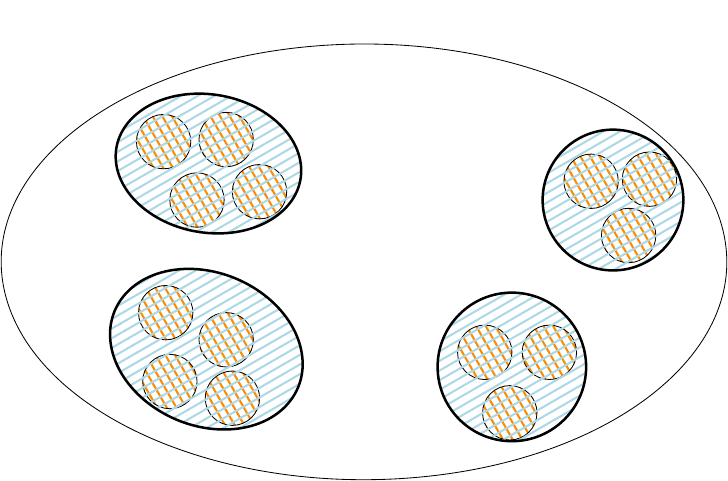}
\includegraphics[scale=0.3,page=6]{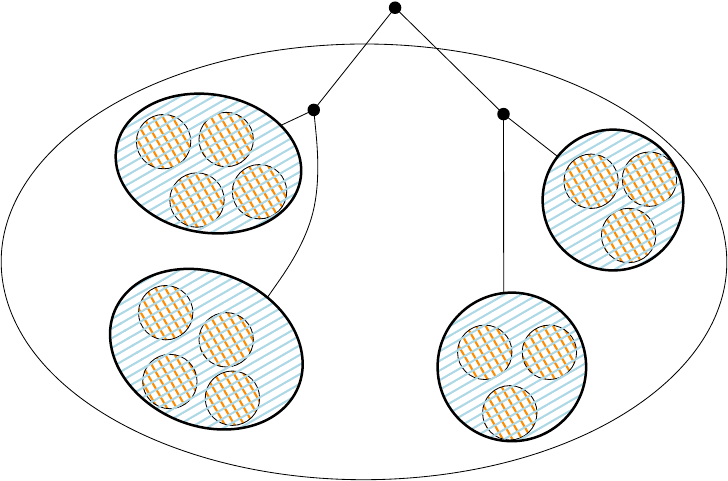}
\caption{A schematic illustration of the proof of
Lemma~\ref{lem:2.5}. The first figure from the left depicts three
levels of the fragmentation map, the middle level being separated.
In the second figure from the left we consider a certain induced
metric (see~\eqref{eq:def d_u}) on the clusters in the middle level.
Due to the separation property, this metric approximates the actual
distances between points in different clusters. In the third  figure
from the left we have applied the weighted finite Dvoretzky theorem,
i.e., Theorem~\ref{lem:theta(D)}, to the middle level clusters, thus
obtaining an appropriately large subset of clusters on which the
induced metric is approximately an ultrametric. The rightmost figure
describes the tree representation of this new ultrametric.}
\label{fig:separation}
\end{minipage}
}
\end{figure}

For every vertex $u\in T\setminus\L(T)$ let $C_u=\p_T^{-1}(u)$ be the set of children
of $u$ in $T$. Let $\bd_u$ be a metric defined on $C_u$ as follows
\begin{equation}\label{eq:def d_u}
 \bd_u(x,y)=\left\{ \begin{array}{ll} \diam_d\left(\left(\partial\F_x\right)\cup\left(\partial\F_y\right)\right)& \mathrm{if}\ x\neq y,\\
 0&\mathrm{if}\ x=y.\end{array}\right.
\end{equation}
The validity of the triangle inequality for $\bd_u$ is immediate to verify.
By the definition of $\theta(D)$ there exists a subset $S_u\subseteq C_u$ such that
\begin{equation} \label{eq:subtree-power}
 \sum_{x\in S_u} w(x)^{\theta(D)}\stackrel{\eqref{eq:theta(D) condition}}{ \ge} \left(\sum_{x\in C_u} w(x)\right)^{\theta(D)} \stackrel{\eqref{eq:sub in compose lem}}{\ge}
w(u)^{\theta(D)}.
\end{equation}
and $(S_u, \bd_u)$ embeds with distortion  $D$ into an  ultrametric space. By  Lemma~\ref{lem:truncation} there exists an ultrametric $\rho_u$ on $S_u$ such that every $x,y\in S_u$ satisfy
\begin{multline}\label{eq:rho_u condition}
\bd_u(x,y) \le \rho_u(x,y) \le \min\left\{\diam_{\bd_u}(S_u), D \bd_u(x,y)\right\}\\\stackrel{\eqref{eq:def d_u}}{\le}\min\left\{\diam_d\left(\bigcup_{x\in S_u}\partial\F_x\right),D \bd_u(x,y)\right\}.
\end{multline}

The subtree $T'\subseteq T$ is now defined inductively in a top-down fashion as follows: declare $r(T)\in T'$ and if $u\in T$ is a non-leaf vertex that was already declared to be in $T'$, add the vertices in $S_u$ to $T'$ as well. Inequality~\eqref{eq:weight-2.5} follows from~\eqref{eq:subtree-power}.
It remains to prove that $\left(\partial\G_{r(T')},d\right)$ embeds into an ultrametric space with distortion $D\left(1+2/\beta\right)$. To this end fix $p,q\in \partial\G_{r(T')}$ and choose the corresponding $a,b\in \L(T')$ such that $\G_a=\{p\}$ and $\G_b=\{q\}$. Let $u=\lca_{T'}(a,b)=\lca_T(a,b)$ and choose $x,y\in S_u$ that are weak ancestors of $a$ and $b$, respectively. Define $\rho(p,q)=\rho_u(x,y)$. Now,
\begin{equation*}\label{eq:rho lower}
d(p,q)=d(\G_a,\G_b)\le \diam_d\left(\left(\partial\F_x\right)\cup\left(\partial\F_y\right)\right)\stackrel{\eqref{eq:def d_u}}{=}\bd_u(x,y)\stackrel{\eqref{eq:rho_u condition}}{\le} \rho_u(x,y)=\rho(p,q).
\end{equation*}
The corresponding lower bound on
$d(p,q)$
is  proved as follows, using the assumption that the fragmentation map $\F$ is $\beta$-separated.
\begin{multline*}\label{eq:rho lower}
\frac{\rho(p,q)}{D}\stackrel{\eqref{eq:rho_u condition}}{\le} \bd_u(x,y)\stackrel{\eqref{eq:def d_u}}{\le} \diam_d\left(\partial\F_x\right)+\diam_d\left(\partial\F_y\right)+d\left(\partial\F_x,\partial\F_y\right)
\stackrel{\eqref{eq:def separated}}\le \left(1+\frac{2}{\beta}\right)d(p,q).
\end{multline*}

 We now argue that $\rho$ is an ultrametric on $\partial\G_{r(T')}$. This is where we will use the truncation in~\eqref{eq:rho_u condition}, i.e., that for all $u\in T\setminus \L(T)$ we have $\diam_{\rho_u}(S_u)\le\diam_d\left(\bigcup_{x\in S_u}\partial\F_x\right).$
Take $p_1,p_2,p_3\in \partial\G_{r(T')}$ and choose the corresponding $a_1,a_2,a_3\in \L(T')$ such that $\G_{a_i}=\{p_i\}$ for $i\in \{1,2,3\}$. By relabeling the points if necessary, we may assume that $u=\lca_T(a_1,a_2)$ is a weak descendant of $v=\lca_T(a_2,a_3)$. If $u=v$ take $x_1,x_2,x_3\in S_u$ that are weak ancestors of $a_1,a_2,a_3$, respectively. Since $\rho_u$ is an ultrametric, it follows that
$$
\rho(p_1,p_2)=\rho_u(x_1,x_2)\le \max\left\{\rho_u(x_1,x_3),\rho_u(x_3,x_2)\right\}= \max\left\{\rho_u(p_1,p_3),\rho_u(p_3,p_2)\right\}.
$$
If, on the other hand, $u$ is a proper descendant of $v$ then choose $x_1,x_2\in S_u$ that are weak ancestors of $a_1,a_2$ (respectively), and choose $s,t\in S_v$ that are weak ancestors of $u,a_3$ (respectively). Then,

\begin{multline*}
\rho(p_1,p_2)=\rho_u(x,y)\stackrel{\eqref{eq:rho_u condition}}{\le} \diam_d\left(\bigcup_{w\in S_u}\partial\F_w\right)\le \diam_d(\partial\F_s)\\\le
\diam_d\left((\partial\F_s)\cup(\partial\F_t)\right)\stackrel{\eqref{eq:def d_u}}{=}\bd_v(s,t)\stackrel{\eqref{eq:rho_u condition}}{\le} \rho_v(s,t)=\rho(p_1,p_3)
=\rho(p_2,p_3).
\end{multline*}
This establishes the ultratriangle inequality for $\rho$, completing the proof of Lemma~\ref{lem:2.5}.
\end{proof}

The next lemma establishes the existence of an intermediate
fragmentation map with useful geometric properties; its proof is
deferred to Section~\ref{sec:proof:after-first-sparsification}.


\begin{lemma} \label{lem:after-first-sparsification} Fix $\tau\in (0,1/3)$ and integers $m,h,k\ge 2$ with $h\ge 2k^2$. Let $(X,d,\mu)$ be  a finite metric measure space of diameter $1$.
 Then there exists a fragmentation map $\F:T\to 2^X$ with the following properties.
\begin{enumerate}[\{1\}]
\item \label{it:0} All the leaves of the tree $T$ are at depth $mh$.
\item \label{it:1} For every $u\in T$ we have
\begin{equation}\label{eq:diam decay cor}
\diam(\F_u)\le \tau^{\depth_T(u)}.
\end{equation}
\item \label{it:2} Denote by $R\subseteq T$ the set of vertices at depths which are integer multiples of $h$. Then for every non-leaf $u\in R$,
\begin{equation} \label{eq:power.(h-1)/h}
 \sum_{v\in D_T(u,R)} \mu\left(\F_v\right)^{\left(1-\frac1{k}\right)^2} \ge \mu\left(\F_u\right)^{\left(1-\frac1{k}\right)^2} .
\end{equation}
Recall that $D_{T}(\cdot,\cdot)$ is given in
Definition~\ref{def:D*}.
\item \label{it:3} There is a subset $S\subseteq T$ containing the root  of $T$ such that
$R$ and $S$ are ``alternating" in the following sense. For every $u,v\in R$ such that $\depth_T(v)=\depth_T(u)+h$ and $v$ is a descendant of $u$, there is one and only one $w\in S$ such that $w$ lies on the path joining $u$ and $v$ and $\depth_T(u)<\depth_T(w)\le \depth_T(v)$.
\item \label{it:4} The vertices of $S$ are $\frac{1-3\tau}{2\tau}$-separated (recall Definition~\ref{def:separated}).
\item \label{it:6} $\F$ is  $\left(\frac{2}{1-3\tau} \tau^{-2h},\tau^{-1}\right)$-lacunary (recall Definition~\ref{def:lacunary}).
\end{enumerate}
\end{lemma}
The vertices of the subset  $R\subseteq T$ of Lemma~\ref{lem:after-first-sparsification} satisfy an inductive inequality~\eqref{eq:power.(h-1)/h} on the measures of their images that will allow us  to (eventually) deduce the covering property~\eqref{eq:lemma goal} of Lemma~\ref{lem:3}. Figure~\ref{fig:RS-frag} contains a schematic depiction of the fact that the levels of $R$ and $S$ alternate.

\begin{figure}[ht]
\centering
\fbox{
\begin{minipage}{6in}
\centering \includegraphics[scale=0.6]{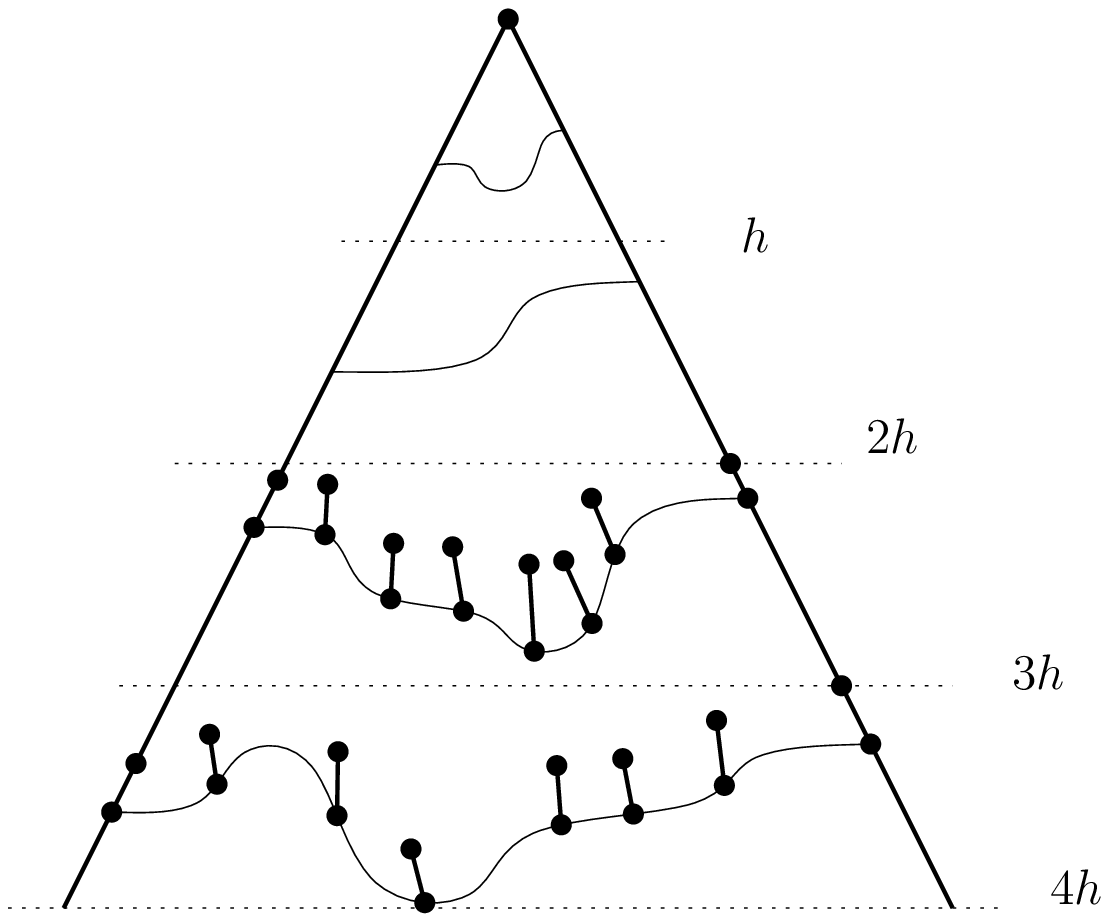}
\caption{A schematic depiction of the tree $T$ corresponding to the fragmentation map $\mathcal F$ of Lemma~\ref{lem:after-first-sparsification}.
The vertices of $R$ are  on the dotted lines. The vertices of $S$ are on the curved solid lines. On every root-leaf path in $T$ the vertices in $R$ and $S$ alternate.
}
\label{fig:RS-frag}
\end{minipage}
}
\end{figure}

We are now in position to prove Lemma~\ref{lem:3} using Lemma~\ref{lem:2.5} and assuming the validity of Lemma~\ref{lem:after-first-sparsification} (recall that Lemma~\ref{lem:after-first-sparsification} will be proved in Section~\ref{sec:proof:after-first-sparsification}).

\begin{proof}[Proof of Lemma~\ref{lem:3}] Let $k,\tau$ be as in Lemma~\ref{lem:3}. Denote $h=2k^2$ and fix $m\in \N$ satisfying
\begin{equation}\label{eq:m big}
\min_{\substack{x,y\in X\\x\neq y}} d(x,y)>\tau^{(m-2)h+1}.
\end{equation}
Apply Lemma~\ref{lem:after-first-sparsification} with the parameters $\tau, m,h,k$ as above, obtaining a fragmentation map $\F:T^1\to 2^X$ with corresponding subsets $S,R\subseteq T^1$. Let $T^2$ be the tree induced by $T^1$ on $S$, i.e., join $u,v\in S$ by an edge of $T^2$ if $u$ is an ancestor of $v$ in $T^1$ and any $w\in S$ that is an ancestor of $v$ in $T^1$ is a weak ancestor of $u$. This is the same as the requirement $v\in D_{T^1}(u,S)$.  Let $\S:T^2\to 2^X$ be the fragmentation map obtained by restricting $\F$ to $S$. To check that $\S$ is indeed a fragmentation map we need to verify that if $u\in \L(T^2)$ then $\F_u$ is a singleton. To see this let $v=\p_{T^2}(u)$ be the parent in $T^2$ of $u$. Since all the leaves of $T$ are at depth $mh$, we have $\depth_{T^1}(v)\ge (m-2)h+1$. Using~\eqref{eq:diam decay cor} we deduce  that $\diam(\F_u)\le \diam(\F_v)\le \tau^{(m-2)h+1}$, implying that $\S_u$ and $\S_v$ are both singletons due to~\eqref{eq:m big}.

Since by Lemma~\ref{lem:after-first-sparsification} we know that the vertices in $S$ are $\frac{1-3\tau}{2\tau}$-separated in the fragmentation map $\F$, it follows that the fragmentation map $\S$ is $\frac{1-3\tau}{2\tau}$-separated. Lemma~\ref{lem:after-first-sparsification} also ensures that $\F$ is $\left(\frac{2}{1-3\tau} \tau^{-2h},\tau^{-1}\right)$-lacunary. This implies that $\S$ is $\left(\frac{2}{1-3\tau} \tau^{-2h},\tau^{-2h}\right)$-lacunary. Indeed, due to Lemma~\ref{lem:after-first-sparsification} we know that if $u\in S$ and $v\in D_{T^1}(u,S)$ is a child of $u$ in $T^2$ then $\depth_{T^1}(v)\le \depth_{T^1}(u)+2h-1$. This implies that if $q,u\in S$ are such that $u$ is a weak descendant of $q$ in $T^2$ then $\depth_{T^1}(u)-\depth_{T^1}(q)\le 2h\left(\depth_{T^2}(u)-\depth_{T^2}(q)\right)$. Hence, if $v,w\in D_{T^1}(u,S)$ are distinct children of $u$ in $T^2$, choose distinct $x,y\in T^1$ that are children of $u$ in $T^1$ and  weak ancestors of $v,w$ (respectively), and use  the fact that $\F$ is $\left(\frac{2}{1-3\tau} \tau^{-2h},\tau^{-1}\right)$-lacunary to deduce that

\begin{multline*}
\diam_d(\S_q)=\diam_d(\F_q)\le \frac{2\tau^{-2h}}{1-3\tau} \cdot\tau^{-\left(\depth_{T^1}(u)-\depth_{T^1}(q)\right)}d\left(\partial\F_x,\partial\F_y\right)\\\le
\frac{2\tau^{-2h}}{1-3\tau}\cdot\tau^{-2h\left(\depth_{T^2}(u)-\depth_{T^2}(q)\right)}d\left(\partial\S_x,\partial\S_y\right).
\end{multline*}

Define  $w_R:R\to (0,\infty)$ by top-down induction as follows. Set
\begin{equation}\label{eq:w_R(r)}
 w_R(r)=\mu(X)^{\left(1-\frac1{k}\right)^2},
\end{equation}
where $r$ is the root of $T^1$. If $u\in R$ is not a leaf and $v\in D_{T^1}(u,R)$ then define
\begin{equation}\label{eq:def w_R}
w_R(v)=\frac{w_R(u)}{\sum_{z\in D_{T^1}(u,R)}\mu(\F_z)^{\left(1-\frac1{k}\right)^2}}\cdot \mu(\F_v)^{\left(1-\frac1{k}\right)^2}.
\end{equation}
Thus for every non-leaf $u\in R$ we have
\begin{equation}\label{eq:equality w_R}
w_R(u)=\sum_{v\in D_{T^1}(u,R)} w_R(v).
\end{equation}
Moreover, it follows from the recursive definition~\eqref{eq:def w_R} combined  with~\eqref{eq:power.(h-1)/h} that
\begin{equation}\label{eq:upper w_R}
\forall\ u\in R,\quad w_R(u)\le \mu(\F_u)^{\left(1-\frac1{k}\right)^2}.
\end{equation}

 Recalling the notation $D_T^*(x,A)$ as given in~\eqref{eq:def D^*}, by summing~\eqref{eq:equality w_R} we see that for all $u\in S\setminus \L(T^2)$ we have
\begin{equation}\label{eq:summed}
\sum_{x\in D^*_{T^1}(u,R)} \sum_{y\in D_{T^1}(x,R)}w_R(y)= \sum_{x\in D^*_{T^1}(u,R)} w_R(x).
\end{equation}
Notice that
\begin{equation}\label{eq:disjoint union}
\bigcup_{x\in D^*_{T^1}(u,R)}D_{T^1}(x,R)=\bigcup_{v\in D_{T^1}(u,S)} D^*_{T^1}(v,R),
\end{equation}
where the unions on both sides of~\eqref{eq:disjoint union} are disjoint. Hence,
\begin{equation}\label{eq:double summation}
\sum_{x\in D^*_{T^1}(u,R)} w_R(x)\stackrel{\eqref{eq:summed}\wedge\eqref{eq:disjoint union}}{=}\sum_{v\in D_{T^1}(u,S)}\sum_{z\in D^*_{T^1}(v,R)}w_R(z).
\end{equation}
Define $w_S:S\to (0,\infty)$ by
\begin{equation}\label{eq:def w S}
w_S(u)=\sum_{z\in D_{T^1}^*(u,R)}w_R(z).
\end{equation}
Then for all $u\in S\setminus \L(T^2)$ we have
$$
\sum_{v\in\p_{T^2}^{-1}(u)}w_S(v)=\sum_{v\in D_{T^1}(u,S)} w_S(v)\stackrel{\eqref{eq:double summation}\wedge\eqref{eq:def w S}}{=} \sum_{x\in D^*_{T^1}(u,R)} w_R(x)
\stackrel{\eqref{eq:def w S}}{=} w_S(u).
$$
This establishes condition~\eqref{eq:sub in compose lem} of Lemma~\ref{lem:2.5} for the weighting $w_S$ of $T^2$. Before applying Lemma~\ref{lem:2.5} we record one more useful fact about $w_S$. Recall that for $u\in S$ the vertex $\p_{T^2}(u)$ is $r$ if $u=r$, and otherwise it is the first proper ancestor of $u$ in $T^1$ which is in $S$.  Take
$u'\in D_{T^1}^*(\p_{T^2}(u),R)$ which is a weak ancestor of $u$ in $T^1$.
Then $\F_{\p_{T^2}(u)}\supseteq \F_{u'}$, and therefore
\begin{multline}\label{eq:w_S better}
\mu\left(\F_{\p_{T^2}(u)}\right)^{\left(1-\frac1{k}\right)^2}\ge \mu\left(\F_{u'}\right)^{\left(1-\frac1{k}\right)^2}\stackrel{\eqref{eq:upper w_R}}{\ge} w_R(u')
\\\stackrel{\eqref{eq:equality w_R}}{=}\sum_{x\in D_{T^1}(u',R)} w_R(x)=\sum_{x\in D_{T^1}^*(u,R)} w_R(x)\stackrel{\eqref{eq:def w S}}{=}w_S(u).
\end{multline}

Apply Lemma~\ref{lem:2.5} to  $\S:T^2\to 2^X$ and  $w_S:T^2\to (0,\infty)$, with $\beta=(1-3\tau)/(2\tau)$ and the parameter $D$ of Lemma~\ref{lem:2.5} replaced by $D/(1+2/\beta)=D(1-3\tau)/(1+\tau)$. Note that our assumption $\tau<(D-2)/(3D+2)$ guarantees that this new value of $D$ is bigger than $2$, so we are indeed allowed to use Lemma~\ref{lem:2.5}. We therefore obtain a subtree $T\subseteq T^2$ with the same root, such that the restricted fragmentation map $\G=\S|_T$ satisfies the following properties.
\begin{itemize}
\item  $(\partial \G_{r(T)},d)$ embeds into an ultrametric space with distortion $D$.
\item Every non-leaf vertex $u\in T\setminus \L(T)$ satisfies
\begin{equation}\label{eq:weight-2.5-use}
\sum_{v\in\p_{T}^{-1}(u)} w_S(v)^{\theta\left(\frac{1-3\tau}{1+\tau}D\right)} \stackrel{\eqref{eq:weight-2.5}}{\ge} w_S(u)^{\theta\left(\frac{1-3\tau}{1+\tau}D\right)}.
\end{equation}
\end{itemize}

Let $G\subseteq T$ be a cut-set of $T$. Define $G_0$ to be a subset of $G$ which is still a cut-set and is minimal with respect to inclusion. Assume inductively that we defined a cut-set $G_i$ of $T$  which is minimal with respect to inclusion. Let $v\in G_i$ be such that $\depth_T(v)$ is maximal. Let $u=\p_T(v)$. By the maximality of $\depth_T(v)$, since $G_i$ is a minimal cut-set of $T$ we necessarily have $\p_T^{-1}(u)\subseteq G_i$, i.e., all the siblings of $v$ in $T$ are also in $G_i$. Note that $G_i'=\left(G_i\cup\{u\}\right)\setminus \p_T^{-1}(u)$ is also a cut-set of $T$, so let $G_{i+1}$ be a subset of $G_i'$ which is still a cut-set of $T$ and is minimal with respect to inclusion. Then,
\begin{multline}\label{eq:to concatenate}
\sum_{v\in G_{i}}w_S(v)^{\theta\left(\frac{1-3\tau}{1+\tau}D\right)} =\sum_{v\in G_i\setminus \p_T^{-1}(u)} w_S(v)^{\theta\left(\frac{1-3\tau}{1+\tau}D\right)}
+\sum_{v\in  \p_T^{-1}(u)} w_S(v)^{\theta\left(\frac{1-3\tau}{1+\tau}D\right)} \\
\stackrel{\eqref{eq:weight-2.5-use}}{\ge} \sum_{v\in G_{i}'}w_S(v)^{\theta\left(\frac{1-3\tau}{1+\tau}D\right)} \ge \sum_{v\in G_{i+1}}w_S(v)^{\theta\left(\frac{1-3\tau}{1+\tau}D\right)}.
\end{multline}
After finitely many iterations of the above process we will arrive at $G_j=\{r\}$. By concatenating the inequalities~\eqref{eq:to concatenate} we see that
\begin{equation}\label{almost}
\sum_{v\in G}w_S(v)^{\theta\left(\frac{1-3\tau}{1+\tau}D\right)}\ge \sum_{v\in G_{0}}w_S(v)^{\theta\left(\frac{1-3\tau}{1+\tau}D\right)}\ge w_S(r)^{\theta\left(\frac{1-3\tau}{1+\tau}D\right)}\stackrel{\eqref{eq:w_R(r)}\wedge\eqref{eq:def w S}}{=}\mu(X)^{\left(1-\frac1{k}\right)^2\theta\left(\frac{1-3\tau}{1+\tau}D\right)}.
\end{equation}
The desired inequality~\eqref{eq:lemma goal} follows from~\eqref{almost} and~\eqref{eq:w_S better}. Since $\S$ is $\frac{1-3\tau}{2\tau}$-separated and $\left(\frac{2}{1-3\tau} \tau^{-4k^2},\tau^{-4k^2}\right)$-lacunary (recall that $h=2k^2$), the same holds true for $\G$ since it is the restriction of $\S$ to the subtree of $T^2$.
\end{proof}

\begin{remark} \label{rem:simplified-argument}
In Theorem~\ref{thm:measure}, if one is willing to
settle for ultrametric distortion $e^{O(1/\e^2)}$, instead of the asymptotically optimal $O(1/\e)$ distortion, then it is possible to simplify Lemma~\ref{lem:3} and its proof. In particular, there is no need to apply Lemma~\ref{lem:2.5}, and consequently also Theorem~\ref{lem:theta(D)}.
Thus one can use the fragmentation map $\mathcal S$ introduced in the proof
of Lemma~\ref{lem:3} as the fragmentation map produced by
Lemma~\ref{lem:3}. Since $\mathcal S$ is $\left(\frac{2}{1-3\tau} \tau^{-4k^2},\tau^{-4k^2}\right)$-lacunary,
Lemma~\ref{lem:sep->ultra} implies that $(\partial \mathcal S_{r(S)},d)$ embeds in an ultrametric space with distortion $\frac{2}{1-3\tau} \tau^{-4k^2}=e^{O(1/\e^2)}$. It is possible to further simplify the proof of the cut-set inequality~\eqref{eq:lemma goal} in the proof of Lemma~\ref{lem:3} by considering a different fragmentation map $\mathcal R$ instead of $\mathcal S$, defined as follows.
Consider the tree $T^3$ induced by $T^1$ on $R$,
and the fragmentation map $\mathcal R:T^3\to 2^X$  obtained by restricting  $\mathcal F$ to $T^3$.
Like $\mathcal S$, the fragmentation map $\mathcal R$ is $\left(\frac{2}{1-3\tau} \tau^{-4k^2},\tau^{-4k^2}\right)$-lacunary, and the proof of~\eqref{eq:lemma goal} for $\mathcal R$
can now be performed by only using the weight function $w_R$, without the need
to consider $w_S$. Unlike $\mathcal S$, the fragmentation map $\mathcal R$
is not separated, and therefore cannot be used with Lemma~\ref{lem:2.5}. However, for the above simplified
argument, Lemma~\ref{lem:2.5} and the separation property are not needed.
\end{remark}

\section{An intermediate fragmentation map: proof of Lemma~\ref{lem:after-first-sparsification}}
\label{sec:proof:after-first-sparsification}

Here we prove Lemma~\ref{lem:after-first-sparsification}.
The proof uses two building blocks:
Lemma~\ref{lem:initiate}, which constructs an initial partition map,
and  Lemma~\ref{lem:sparse-s-2}, which prunes a given weighted rooted tree.
The basic idea of the proof Lemma~\ref{lem:after-first-sparsification} can be described as follows.
Lemma~\ref{lem:initiate} constructs an initial partition map together  with a ``designated child"
for every non leaf vertex. The designated children have, roughly speaking, the largest weight among
their siblings, and they are also pairwise separated. The pruning step of
Lemma~\ref{lem:sparse-s-2} can now focus on the combinatorial structure of the partition map, pruning
the associated tree so as to keep at some levels only the designated children of the level above.
This guarantees the separation property as well as the desired estimate~\eqref{eq:power.(h-1)/h}.


The exact notion of ``size" used to choose  designated children
is tailored to be compatible with the
ensuing pruning step,
and is the content of the following definition.
Observe that any fragmentation map $\F:T\to 2^{X}$ induces a
weighting $w:T\to (0,\infty)$ of the vertices of $T$ given by
$w(u)=\mu(\F_u)$.
For our purpose, we will need a modified version of  it,
described in the following definition.
%
%
%
%

\begin{definition}[Modified weight function]\label{def:w modified}
Fix integers  $h,k\ge 2$ and let $T$ be a finite graph-theoretical
rooted tree, all of whose leaves are at the same depth, which is
divisible by $h$. Assume that we are given $w:T\to (0,\infty)$.
Define a new function $w_h^k:T\to (0,\infty)$ as follows. If $u\in
\L(T)$ then
$$
w_h^k(u)=w(u)^{\frac{k-1}{k}}.
$$
Continue defining $w_h^k(u)$ by reverse induction on $\depth_T(u)$ as
follows.
\begin{equation}\label{eq:w_k}
w_h^k(u)=\left\{\begin{array}{ll}w(u)^{\frac{k-1}{k}} &
\mathrm{if}\ h \mid\depth_T(u),\\
\sum_{v\in \p^{-1}(u)} w_h^k(v)& \mathrm{if}\ h\nmid
\depth_T(u).\end{array}\right.
\end{equation}
Equivalently, if $u\in T$ and $(j-1)h<\depth_T(u)\le jh$ for some
integer $j$ then
\begin{equation}\label{eq:explicit w_h^k}
w_h^k(u)=\sum_{\substack{v\in T_u\\ \depth_T(v)=jh}}
w(v)^{\frac{k-1}{k}}.
\end{equation}
\end{definition}

\begin{lemma}\label{lem:initiate}
Let $(X,d,\mu)$ be a finite metric measure space of diameter $1$ and
$\tau\in (0,1/3)$. For every triple of integers $m,h,k\ge 2$ there exists a
fragmentation map $\F:T\to 2^X$ with the following properties.
\begin{itemize}
\item All the leaves of the tree $T$ are at depth $mh$.
\item $\F$ is a partition map, i.e., $\partial F_{r(T)}=X$.
\item For every $u\in T$ we have
\begin{equation}\label{eq:diam decay}
\diam(\F_u)\le \tau^{\depth_T(u)}.
\end{equation}
\item Every non-leaf vertex $u\in T\setminus \L(T)$ has a ``designated child" $\dc(u)\in \p^{-1}(u)$ such that
\begin{equation}\label{eq:designated}
w_h^k(\dc(u))=\max_{v\in \p^{-1}(u)} w_h^k\left(v\right),
\end{equation}
where $w:T\to (0,\infty)$ is given by $w(u)=\mu(\F_u)$ and $w_h^k:T\to
(0,\infty)$ is the modified weight function from
Definition~\ref{def:w modified}.
\item Suppose that $u,v\in T\setminus\L(T)$ satisfy $\depth_T(u)=\depth_T(v)$ and $u\neq v$. Then
\begin{equation}\label{eq:sep designated}
d\left(\F_{\dc(u)},\F_{\dc(v)}\right)>\frac{1-3\tau}{2}\cdot \tau^{\depth_T(u)}.
\end{equation}
\end{itemize}
\end{lemma}

A schematic description of the partition map that is constructed in Lemma~\ref{lem:initiate} is depicted in Figure~\ref{fig:initial-part}. Lemma~\ref{lem:initiate}  will be proved in Section~\ref{sec:initilize}.

\begin{figure}[ht]
\centering
\fbox{
\begin{minipage}{6in}
\centering \includegraphics[scale=0.9]{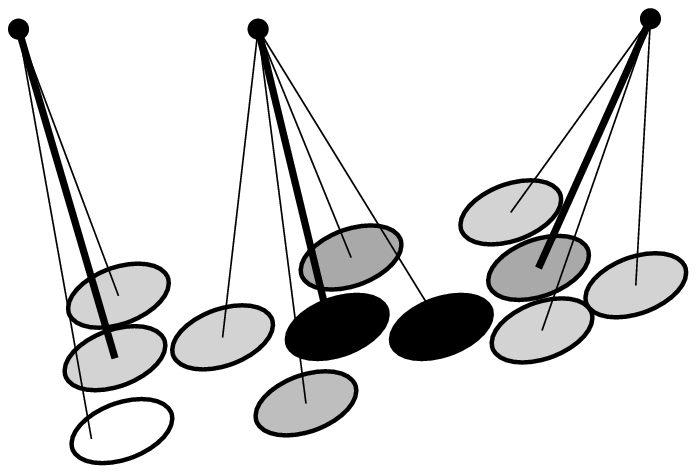}

\caption{A schematic depiction of two levels in the initial partition map that is constructed in Lemma~\ref{lem:initiate}. In the lower level the darkness of the cluster represent their weight $w_h^k$; darker means larger weight.
A thick line represents the designated child of the higher level cluster.
Notice that the designated child is the  child of its parent of largest weight,
and that the designated children are far from each other.}
\label{fig:initial-part}
\end{minipage}
}
\end{figure}


Lemma~\ref{lem:sparse-s-2} below is the pruning step.
The appropriate setting for the
pruning is a certain class of weighted trees, which we now
introduce; the definition below contains the tree of
Lemma~\ref{lem:initiate} as a special case.

\begin{definition}[Subadditive weighted tree with designated
children]\label{def:subadditive}
 Fix integers $h,k\ge 2$. A subadditive weighted tree with designated children is a triple $(T,w,\dc)$
 consisting of a finite rooted graph-theoretical tree $T$, a mapping
$w:T\to (0,\infty)$ and for every non-leaf vertex $u\in T\setminus
\L(T)$ a ``designated child" $\dc(u)\in \p^{-1}(u)$, such that the
following conditions hold true.
\begin{itemize}
\item All the leaves of $T$ are at the same depth, which is divisible by $h$.
\item For every non-leaf vertex $u\in T\setminus \L(T)$,
\begin{equation}\label{eq:subaddtitive}
w(u)\le\sum_{v\in \p^{-1}(u)} w(v).
\end{equation}
\item For every non-leaf vertex $u\in T\setminus \L(T)$,
\begin{equation}\label{eq:max c}
w_h^k(\dc(u))=\max_{v\in \p^{-1}(u)}w_h^k(v),
\end{equation}
where $w_h^k:T\to (0,\infty)$ is the modified weight function of Definition~\ref{def:w modified}.
\end{itemize}
\end{definition}

\begin{definition}[Subtree sparsified at a subset]
\label{def:spars at S}
Fix two integers $h,k\ge 2$ and let let $(T,w,\dc)$ be a subadditive
weighted tree with designated children (recall
Definition~\ref{def:subadditive}). Let $T'$ be a subtree of $T$ (see
Definition~\ref{def:subtree}) and $S\subseteq T'$. We say that the
subtree $T'$ is sparsified at $S$ if every $v\in S$ is a designated
child and has no siblings in $T'$. For the purpose of this
definition we declare the root $r(T)$ to be a designated child,
i.e., we allow $r(T)\in S$. Thus $v\in T$ is a designated child if
it is either the root of $T$ or $\dc(\p(v))=v$.
\end{definition}

\begin{lemma}
\label{lem:sparse-s-2}
Fix two integers $h,k\ge 2$ with
$h\ge 2k^2$.
Let $(T,w,\dc)$ be a subadditive
weighted tree with designated children as in
Definition~\ref{def:subadditive} (thus all the leaves of $T$ are at
the same depth, which is divisible by $h$, and the designated child
map $\dc$ satisfies~\eqref{eq:max c}). Then there exists a subtree
$T'$ of $T$ with the same root as $T$, and two subsets
$R,S\subseteq T'$, both containing the root of $T'$,
with the following properties:
\begin{compactitem}
\item For any non-leaf $u\in T'$ we have $\dc(u)\in T'$.
\item $R=\{v\in T': h\mid \depth_{T}(v)\}$. 
\item For any non-leaf vertex $u\in R$,
\begin{equation} \label{eq:power-kids}
 \sum_{v\in D_{T'}(u,R)} w(v)^{\left(1-\frac{1}{k}\right)^2} \ge w(u)^{\left(1-\frac{1}{k}\right)^2} .
 \end{equation}
Recall that $D_{T'}(\cdot,\cdot)$ is given in
Definition~\ref{def:D*}.
\item For every $u,v\in R$ such that $\depth_T(v)=\depth_T(u)+h$ and $v$ is a descendant of $u$, there is one and only one $w\in S$ such that $w$ lies on the path joining $u$ and $v$ and $\depth_T(u)<\depth_T(w)\le \depth_T(v)$.
\item For any $u\in T'$ such that $D_{T'}(u,S)\neq\emptyset$, all the vertices of $D_{T'}(u,S)$ are at the same depth in $T'_u$,
which is an integer between $1$ and $2h$.
\item $T'$ is sparsified at the subset $S$.
\end{compactitem}
\end{lemma}

Lemma~\ref{lem:sparse-s-2} will be proved in Section~\ref{sec:holder}. Assuming the validity of Lemma~\ref{lem:sparse-s-2}, as well as the validity of Lemma~\ref{lem:initiate} (which will be proved is Section~\ref{sec:initilize}), we  are now in position to deduce Lemma~\ref{lem:after-first-sparsification}.

\begin{proof} [Proof of Lemma~\ref{lem:after-first-sparsification}]
Let $\F^0:T^0\to 2^X$ be the partition map of Lemma~\ref{lem:initiate}, constructed with parameters $m,h,k$, and having the associated designated child map $\dc$ from Lemma~\ref{lem:initiate}. Let $T$ be the tree obtained by applying Lemma~\ref{lem:sparse-s-2} to $(T^0,w,\dc)$, where  $w:T^0\to (0,\infty)$ is given by $w(v)=\mu(\F^0_v)$. Define a fragmentation map $\F:T\to 2^X$ by $\F=\F^0|_T$, i.e., by restricting $\F^0$ to the subtree $T$. Properties~\{\ref{it:0}\}, \{\ref{it:1}\} are satisfied by $\F^0$ due to Lemma~\ref{lem:initiate}, and therefore they are also satisfied by $\F$ since $T$ has the same root as $T^0$.
Properties \{\ref{it:2}\}, \{\ref{it:3}\}, 
are part of the conclusion of Lemma~\ref{lem:sparse-s-2}. It remains
to prove properties~\{\ref{it:4}\} and \{\ref{it:6}\}.

Assume that $u\in S$. Take $y\in \left(\partial \F_{r(T)}\right)\setminus \left(\partial \F_u \right)$. In order to prove property~\{\ref{it:4}\}
it suffices to show that $d(y,\F_u)\ge \frac{1-3\tau}{2\tau}\diam(\F_u)$.
By property \{\ref{it:3}\} it follows that $u=\dc(\p(u))$.  Let $w=\lca_T (u,y)$ and take $u',y'\in D_T(w,S)$ such that $u'$ is a weak ancestor of $u$ and $y'$ is a weak ancestor of $y$. By Lemma~\ref{lem:sparse-s-2} we know that $\depth_T(u')=\depth_T(y')$, and therefore by conclusion~\eqref{eq:sep designated} of Lemma~\ref{lem:initiate} and using the fact that $\dc(\p(u'))=u'$ and $\dc(\p(y'))=y'$ (because $u',y'\in S$),
$$
d\left(y,\F_u\right)\ge d\left(\F_{y'},\F_{u'}\right)\ge \frac{1-3\tau}{2}\tau^{\depth_T(u')-1}\ge \frac{1-3\tau}{2\tau}\tau^{\depth_T(u)}\stackrel{\eqref{eq:diam decay cor}}{\ge} \frac{1-3\tau}{2\tau}\diam(\F_u).
$$

It remains to prove property~\{\ref{it:6}\}. Take $q\in T$ and let $u\in T$ be a weak descendent of $q$ that has at least two children in $T$, i.e., $v,w\in \p^{-1}(u)\cap T$, $v\neq w$. Our goal is to show that
\begin{equation}\label{eq:goal lacunary}
\diam\left(\F_q\right)\le \frac{2\tau^{-2h}}{1-3\tau}\cdot\tau^{\depth_T(q)-\depth_T(u)} \cdot d\left(\partial\F_v,\partial\F_w\right).
\end{equation}
Since $v$ and $w$ are siblings in $T$ we know by \{\ref{it:3}\} that 
$\{v,w\}\cap S=\emptyset$.
Note that
$$
\partial F_v=\bigcup_{y\in D(v,S)} \partial\F_y\quad\mathrm{and}\quad \partial F_w=\bigcup_{x\in D(w,S)} \partial\F_x,
$$
and therefore
\begin{equation}\label{eq:boundary dist formula}
d\left(\partial\F_v,\partial F_w\right)=\min_{\substack{y\in D(v,S)\\ x\in D(w,S)}} d\left(\partial F_y,\partial F_x\right).
\end{equation}
Note that since $\{v,w\}\cap S=\emptyset$ we have $D(v,S)\cup D(w,S)\subseteq D(u,S)$.  By Lemma~\ref{lem:sparse-s-2} it follows that all the vertices in $D(v,S)\cup D(w,S)$ are at the same depth in $T$. Denote this depth by $\ell$. Due to Lemma~\ref{lem:sparse-s-2}  we know that $\ell\le \depth_T(u)+2h$. By conclusion~\eqref{eq:sep designated} of Lemma~\ref{lem:initiate} we deduce that for all $y\in D(v,S)$ and $x\in D(w,S)$ we have
\begin{equation}\label{eq:lower boundary dist}
d\left(\partial\F_y,\partial\F_x\right)\ge d\left(\F_y,\F_x\right)=d\left(\F_{\dc(\p(y))},\F_{\dc(\p(x))}\right)>\frac{1-3\tau}{2}\tau^\ell\ge \frac{1-3\tau}{2}\tau^{\depth_T(u)+2h}.
\end{equation}
Now, the desired inequality \eqref{eq:goal lacunary} is proved as
follows.
\begin{equation*}
d\left(\partial\F_v,\partial F_w\right)
\stackrel{\eqref{eq:boundary dist formula}\wedge\eqref{eq:lower boundary dist}}{>}
 \frac{1-3\tau}{2}\tau^{\depth_T(u)+2h}\stackrel{\eqref{eq:diam decay
 cor}}{\ge}
 \frac{1-3\tau}{2}\tau^{\depth_T(u)-\depth_T(q)+2h}\diam(\F_q).\qedhere
\end{equation*}
\end{proof}

\section{The initial fragmentation map: proof of Lemma~\ref{lem:initiate}}
\label{sec:initilize}


\begin{proof}[Proof of Lemma~\ref{lem:initiate}]
The construction of the initial fragmentation map
will be in a bottom-up fashion: the tree $T$ will
be decomposed as a disjoint union on ``levels"
$V_0,V_1,\ldots,V_{mh}$, where $V_i$ are the vertices at depth $i$.
We will construct these levels $V_i$ and the mappings $\F:V_i\to
2^X$ and $w_h^k:V_i\to (0,\infty)$ by reverse induction on $i$, and
describe inductively for each $v\in V_{i+1}$ its parent $u\in V_i$,
as well as the designated child $\dc(u)$. At the end of this
construction $V_0$ will consist of a single vertex, the root of $T$.

Define $\ell_{mh}=|X|$ and write $X=\{x_1,\ldots,x_{\ell_{mh}}\}$.
The initial level $V_{mh}$ consists of the leaves of $T$, and it is
defined to be $V_{mh}=\{v^{mh}_j\}_{j=1}^{\ell_{mh}}$. For all $j\in
\{1,\ldots,\ell_{mh}\}$ we also set
$$
\F_{v^{mh}_j}=\{x_j\}\quad \mathrm{and}\quad
w_h^k(v_j^{mh})=w(x_j)^{\frac{k-1}{k}}=\mu(\{x_j\})^{\frac{k-1}{k}}.$$
Assume inductively that for $i\in \{1,\ldots,mh-1\}$ we have already
defined
$$
V_{i+1}=\left\{v_1^{i+1},v_2^{i+1},\ldots,v_{\ell_{i+1}}^{i+1}\right\},
$$
and the mappings $\F:V_{i+1}\to 2^X$ and $w_h^k:V_{i+1}\to (0,\infty)$.

Choose $j_1\in \{1,\ldots,\ell_{i+1}\}$ such that
$$
w_h^k\left(v_{j_1}^{i+1}\right)=\max_{j\in \{1,\ldots,\ell_{i+1}\}}
w_h^k\left(v_{j}^{i+1}\right).
$$
Define
$$
A_1^i=\left\{s\in \{1,\ldots,\ell_{i+1}\}:\
d\left(\F_{v_{j_1}^{i+1}},\F_{v_{s}^{i+1}}\right)\le
\frac{1-3\tau}{2}\tau^i\right\}.
$$
Create a new vertex $v_1^i\in V_i$ and define
$$
\F_{v_1^i}=\bigcup_{s\in A_1^i} \F_{v^{i+1}_s}.
$$
Also, declare the vertices $\{v^{i+1}_s\}_{s\in A_1^i}\subseteq
V_{i+1}$ to be the children of $v_1^i$, and in accordance
with~\eqref{eq:w_k} define
$$
w_h^k\left(v_1^i\right)=\left\{\begin{array}{ll}w\left(v_1^i\right)^{\frac{k-1}{k}}
&
\mathrm{if}\ h \mid i,\\
\sum_{s\in A_1^i} w_h^k\left(v_s^{i+1}\right)& \mathrm{if}\ h\nmid
i.\end{array}\right.
$$
Finally, set
$$
\dc\left(v_1^i\right)=v_{j_1}^{i+1}.
$$

Continuing inductively, assume that we have defined
$v_1^i,v_2^i,\ldots,v_z^i\in V_i$, together with nonempty disjoint
sets
$$A_1^i,\ldots,A_z^i\subseteq \{1,\ldots,\ell_{i+1}\}.$$
If $\bigcup_{t=1}^z A_t^i=\{1,\ldots,\ell_{i+1}\}$ then define
$\ell_i=z$ and $V_i=\left\{v_1^i,v_2^i,\ldots,v_{z}^i\right\}$.
Otherwise, choose $j_{z+1}\in \{1,\ldots,\ell_{i+1}\}\setminus
\bigcup_{t=1}^z A_t^i$ such that
\begin{equation}\label{eq:mu max}
w_h^k\left(v_{j_{z+1}}^{i+1}\right)=\max_{j\in
\{1,\ldots,\ell_{i+1}\}\setminus \bigcup_{t=1}^z A_t^i}
w_h^k\left(v_{j}^{i+1}\right),
\end{equation}
and define
\begin{equation}\label{eq:def A r+1}
A_{z+1}^i=\left\{s\in \{1,\ldots,\ell_{i+1}\}\setminus
\bigcup_{t=1}^z A_t^i:\
d\left(\F_{v_{j_{z+1}}^{i+1}},\F_{v_{s}^{i+1}}\right)\le
\frac{1-3\tau}{2}\tau^i\right\}.
\end{equation}
Create a new vertex $v_{z+1}^i\in V_i$ and define
\begin{equation}\label{eq:def F r+1}
\F_{v_{z+1}^i}=\bigcup_{s\in A_{z+1}^i} \F_{v^{i+1}_s}.
\end{equation}
Also, declare the vertices $\{v^{i+1}_s\}_{s\in A_{z+1}^i}\subseteq
V_{i+1}$ to be the children of $v_{z+1}^i$ and define
$$
w_h^k\left(v_{z+1}^i\right)=\left\{\begin{array}{ll}w\left(v_{z+1}^i\right)^{\frac{k-1}{k}}
&
\mathrm{if}\ k \mid i,\\
\sum_{s\in A_{z+1}^i} w_h^k\left(v_s^{i+1}\right)& \mathrm{if}\
k\nmid i.\end{array}\right.
$$
Finally, set
\begin{equation}\label{eq:def dc}
\dc\left(v_{z+1}^i\right)=v_{j_{z+1}}^{i+1}.
\end{equation}

The above recursive procedure must terminate, yielding the level $i$
set $V_i$. We then proceed inductively until the set $V_1$ has been
defined. We conclude by defining $V_0$ to be a single new vertex
$r(T)$ (the root) with all the vertices in $V_1$ its children. The
designated child of the root, $\dc(r(T))$, is chosen to be a vertex
$u\in V_1$ such that $w_h^k\left(u\right)=\max_{v\in
V_1}w_h^k\left(v\right)$. We also set $\F_{r(T)}=X$ and
$w_h^k(r(T))=\mu(X)^{\frac{k-1}{k}}$.

The resulting fragmentation map $\F:T\to 2^X$ is by definition a
partition map, since $\F(V_{mh})=\F(\L(T))=X$. Also, the
construction above guarantees the validity of~\eqref{eq:designated}
due to~\eqref{eq:mu max} and~\eqref{eq:def dc}.

We shall now prove~\eqref{eq:diam decay} by reverse induction on
$\depth_T(u)$. If $\depth_T(u)=mh$ then $\diam(\F_u)=0$ and there is
nothing to prove. Assuming the validity of~\eqref{eq:diam decay}
whenever $\depth_T(u)=i+1$, suppose that $\depth_T(u)=i$ and
moreover that $u=v^i_{z+1}$ in the above construction. By virtue
of~\eqref{eq:def A r+1} and~\eqref{eq:def F r+1} we know that
$$
\diam\left(\F_u\right)=\diam\left(\F_{v_{z+1}^i}\right)\le
3\max_{s\in
A_{z+1}^i}\diam\left(\F_{v_s}^{i+1}\right)+2\frac{1-3\tau}{2}\tau^i
\le 3\tau^{i+1}+(1-3\tau)\tau^{i}=\tau^i.
$$
Since~\eqref{eq:diam decay} is also valid for $i=0$ (because
$\diam(X)=1$), this concludes the proof of~\eqref{eq:diam decay}.

It remains to prove~\eqref{eq:sep designated}. Since we are assuming
that $u\neq v$ are non-leaf vertices and $\depth_T(u)=\depth_T(v)$,
we may write  $u=v_s^i$ and $v=v_t^i$ for some $i\in \{1,\ldots,
mh-1\}$ and $s<t$. Then by the above construction
$\dc\left(v_s^i\right)= v_{j_s}^{i+1}$, $\dc\left(v_t^i\right)=
v_{j_t}^{i+1}$ and
 $$
j_t\in \{1,\ldots,\ell_{i+1}\}\setminus
\bigcup_{\ell=1}^{t-1}A_\ell^i\subseteq
\{1,\ldots,\ell_{i+1}\}\setminus \bigcup_{\ell=1}^{s-1}A_\ell^i,
 $$
 yet $j_t\notin A_s^i$. The validity of~\eqref{eq:sep designated} now follows from the definition of $A_s^{i}$; see~\eqref{eq:def A r+1}.
\end{proof}


\section{An iterated H\"older argument for trees: proof of Lemma~\ref{lem:sparse-s-2}}
\label{sec:holder}

Our goal here is to prove Lemma~\ref{lem:sparse-s-2}. The heart of
this lemma is the extraction of a ``large" and ``sparsified" subtree
from any ``subadditive weighted tree with designated children"
(Definition~\ref{def:subadditive}). The resulting tree is depicted
in Figure~\ref{fig:RS-frag-2}.

\begin{figure}[ht]
\centering \fbox{
\begin{minipage}{6in}
\centering \includegraphics[scale=0.6,page=2]{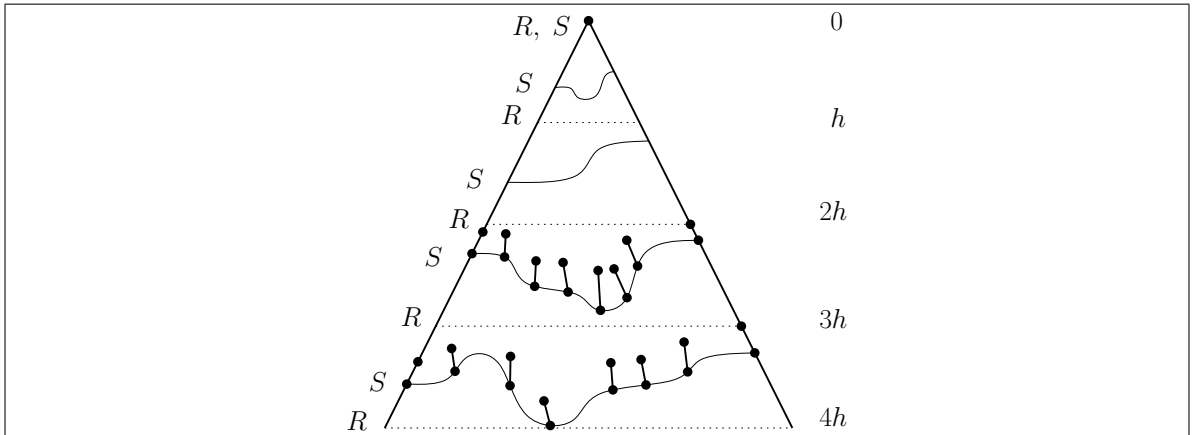}
\caption{A schematic depiction of the subtree $T'$ constructed in Lemma~\ref{lem:sparse-s-2}. The vertices of $R$ are
on the dotted lines and the vertices of $S$ are on the curved solid
lines. The vertices of
$S$ are $\beta$-separated. This is achieved by pruning all their
siblings, leaving each of them as the single offspring of its parent.}
\label{fig:RS-frag-2}
\end{minipage}
}
\end{figure}

%


\begin{definition}[Sparsified tree]\label{def:sparsified} Fix integers $h,k\ge 2$ and
let $(T,w,\dc)$ be a subadditive weighted tree with designated children. For $i\in \Z$ define a subtree $T^{(i)}$ of $T$ as follows.
\begin{equation}\label{eq:T^(i)}
T^{(i)}=T\setminus \left(\bigcup_{\substack{u\in T\\\depth_T(u)=i-1}}\bigcup_{v\in \p^{-1}(u)\setminus\{\dc(u)\}}T_v\right).
\end{equation}
 Thus $T^{(i)}$ is obtained from $T$ by removing all the subtrees rooted at vertices of depth $i$ that are not designated children. Note that by definition $T^{(i)}=T$ if either $i\le 0$ or $T$ has no vertices at depth $i$.
\end{definition}

 Lemma~\ref{lem:holder} below is
inspired by an argument in~\cite[Lem. 3.25]{BLMN05}, though our
assumptions, proof, and conclusion are different.

\begin{lemma}\label{lem:holder}
Fix $h,k\in \N$ satisfying $h\ge k\ge 2$ and let $(T,w,\dc)$ be a subadditive weighted tree with designated children. Assume that all the leaves of $T$ are at depth $h$. Then there exists $L\subseteq \{1,\ldots,h\}$ with $|L|\ge h-k+1$ such that for every $i\in L$ we have
$$
\sum_{\ell\in \L\left(T^{(i)}\right)} w(\ell)^{\frac{k-1}{k}}\ge w(r)^{\frac{k-1}{k}},
$$
where $r=r(T)$ is the root of $T$ and $T^{(i)}$ is as in Definition~\ref{def:sparsified}.
\end{lemma}

\begin{proof} For $i\in \{1,\ldots,h\}$ and $u\in T$ define $f_i(u)\in (0,\infty)$ by reverse induction on $\depth_T(u)$ as follows. If $\depth_T(u)=h$, i.e., $u$ is a leaf of $T$, set
\begin{equation}\label{eq:def f_i leaf}
f_i(u)=w(u)^{\frac{k-1}{k}}.
\end{equation}
If $\depth_T(u)<h$ define recursively
\begin{equation}\label{def:f_i recurse}
f_i(u)=\left\{\begin{array}{ll} \max_{v\in \p^{-1}(u)} f_i(v) & \mathrm{if}\ i=\depth_T(u)+1,\\
\sum_{v\in \p^{-1}(u)} f_i(v)& \mathrm{if}\ i\neq \depth_T(u)+1.\end{array}\right.
\end{equation}

We observe that for all $i\in \{1,\ldots,h\}$ and $u\in T$ we have
\begin{equation}\label{eq:closed fourmula f_11}
i\le \depth_T(u)\implies f_i(u)=\sum_{\ell\in \L(T_u)} w(\ell)^{\frac{k-1}{k}},
\end{equation}
and
\begin{equation}\label{eq;closed formula f_i}
i> \depth_T(u)\implies f_i(u)=\sum_{\ell\in \L\left((T_u)^{(i)}\right)}w(\ell)^{\frac{k-1}{k}},
\end{equation}
where analogously to~\eqref{eq:T^(i)} we define
\begin{equation}\label{eq:T_u^(i)}
(T_u)^{(i)}=T_u \setminus\left(\bigcup_{\substack{a\in T\\\depth_T(a)=i-1}}\bigcup_{b\in \p^{-1}(a)\setminus\{\dc(a)\}}T_b\right).
\end{equation}
In other words, recalling that $T_u$ is the subtree of $T$ rooted at $u$, the subtree $(T_u)^{(i)}$  is obtained from $T_u$ by deleting all the subtrees rooted at vertices of depth $i$ that are not designated children (here depth is measured in $T$, i.e., the distance from the original root $r(T)$).

Identities~\eqref{eq:closed fourmula f_11} and~\eqref{eq;closed formula f_i} follow by reverse induction on $\depth_T(u)$ from the recursive definition of $f_i(u)$. Indeed, if $\depth_T(u)=h$ then~\eqref{eq;closed formula f_i} is vacuous and~\eqref{eq:closed fourmula f_11}  follows from~\eqref{eq:def f_i leaf}. Assume that $u\in T$ is not a leaf of $T$ and that~\eqref{eq:closed fourmula f_11} and~\eqref{eq;closed formula f_i} hold true for the children of $u$. If $i\le \depth_T(u)$ then by~\eqref{def:f_i recurse} and the inductive hypothesis we have
$$
f_i(u)=\sum_{v\in \p^{-1}(u)}f_i(v)= \sum_{v\in \p^{-1}(u)} \sum_{\ell\in \L(T_v)} w(\ell)^{\frac{k-1}{k}}=\sum_{\ell\in \L(T_u)} w(\ell)^{\frac{k-1}{k}}.
$$
If $i=\depth_T(u)+1$ then since we are assuming that~\eqref{eq:closed fourmula f_11} holds for each $v\in \p^{-1}(u)$,
\begin{multline*}
f_i(u)\stackrel{\eqref{def:f_i recurse}}{=}\max_{v\in \p^{-1}(u)} f_i(v)= \max_{v\in \p^{-1}(u)} \sum_{\ell\in \L(T_v)} w(\ell)^{\frac{k-1}{k}}
\stackrel{\eqref{eq:explicit w_h^k}}{=} \max_{v\in \p^{-1}(u)} w_h^k(v)\\\stackrel{\eqref{eq:max c}}{=}w_h^k(\dc(u))\stackrel{\eqref{eq:explicit w_h^k}}{=}
\sum_{\ell\in \L(T_{\dc(u)})} w(\ell)^{\frac{k-1}{k}}\stackrel{\eqref{eq:T_u^(i)}}{=}
\sum_{\ell\in \L\left((T_u)^{(i)}\right)}w(\ell)^{\frac{k-1}{k}}.
\end{multline*}
Finally, if $i>\depth_T(u)+1$ then we are assuming that~\eqref{eq;closed formula f_i} holds for each $v\in \p^{-1}(u)$, and therefore
$$
f_i(u)\stackrel{\eqref{def:f_i recurse}}{=}\sum_{v\in \p^{-1}(u)} f_i(v)\stackrel{\eqref{eq;closed formula f_i}}{=}\sum_{v\in \p^{-1}(u)} \sum_{\ell\in \L\left((T_v)^{(i)}\right)}w(\ell)^{\frac{k-1}{k}}\stackrel{\eqref{eq:T_u^(i)}}{=}
\sum_{\ell\in \L\left((T_u)^{(i)}\right)}w(\ell)^{\frac{k-1}{k}}.
$$
This completes the inductive verification of the identities~\eqref{eq:closed fourmula f_11} and~\eqref{eq;closed formula f_i}.

Our next goal is to prove by reverse induction on $\depth_T(u)$ that for every $H\subseteq \{1,\ldots,h\}$ with $|H|=k$ we have,
\begin{equation}\label{eq:product goal}
\prod_{i\in H} f_i(u)\ge w(u)^{k-1}.
\end{equation}
Indeed, if $\depth_T(u)=h$ then~\eqref{eq:product goal} holds as equality due to~\eqref{eq:def f_i leaf}. Assume inductively that $\depth_T(u)<h$ and that~\eqref{eq:product goal} holds for all the children of $u$. We claim that there exists $j\in H$ such that
\begin{equation}\label{eq:find j}
\prod_{i\in H} f_i(u)\ge \left(\max_{v\in \p^{-1}(u)} f_j(v)\right)\prod_{i\in H\setminus\{j\}} \left(\sum_{v\in \p^{-1}(u)} f_i(v)\right).
\end{equation}
Indeed, if $\depth_T(u)+1\in H$ then take $j=\depth_T(u)+1$ and note that~\eqref{eq:find j} holds as equality due to~\eqref{def:f_i recurse}. On the other hand, if $\depth_T(u)+1\notin H$ then let $j$ be an arbitrary element of $H$, and note that due to~\eqref{def:f_i recurse} we have
\begin{multline*}
\prod_{i\in H} f_i(u)=\left(\sum_{v\in \p^{-1}(u)} f_j(v)\right)\prod_{i\in H\setminus\{j\}} \left(\sum_{v\in \p^{-1}(u)} f_i(v)\right)\\\ge \left(\max_{v\in \p^{-1}(u)} f_j(v)\right)\prod_{i\in H\setminus\{j\}} \left(\sum_{v\in \p^{-1}(u)} f_i(v)\right),
\end{multline*}
as required. Now,
\begin{multline*}
\prod_{i\in H} f_i(u)\stackrel{(*)}{\ge} \left(\max_{v\in \p^{-1}(u)} f_j(v)\right)\left(\sum_{v\in \p^{-1}(u)} \prod_{i\in H\setminus\{j\}} f_i(v)^{\frac{1}{k-1}}\right)^{k-1}\\\ge \left(\sum_{v\in \p^{-1}(u)} \prod_{i\in H} f_i(v)^{\frac{1}{k-1}}\right)^{k-1}
\stackrel{(**)}{\ge} \left(\sum_{v\in \p^{-1}(u)} w(v)\right)^{k-1}\stackrel{\eqref{eq:subaddtitive}}{\ge} w(u)^{k-1},
\end{multline*}
where in $(*)$ we used~\eqref{eq:find j} and H\"older's inequality, and in $(**)$ we used the inductive hypothesis. This concludes the proof of~\eqref{eq:product goal}.

We are now in position to complete the proof of Lemma~\ref{lem:holder}. Set $H_1=\{1,\ldots,k\}$. Then
$$
\max_{i\in H_1} \sum_{\ell\in \L\left(T^{(i)}\right)} w(\ell)^{\frac{k-1}{k}}\ge \left(\prod_{i\in H_1} \sum_{\ell\in \L\left(T^{(i)}\right)} w(\ell)^{\frac{k-1}{k}}\right)^{\frac1{k}} \stackrel{\eqref{eq;closed formula f_i}}{=} \left(\prod_{i\in H_1} f_i(r)\right)^{\frac{1}{k}}\stackrel{\eqref{eq:product goal}}{\ge} w(r)^{\frac{k-1}{k}}.
$$
Hence there is $i_1\in H_1$ satisfying
$$
\sum_{\ell\in \L\left(T^{(i_1)}\right)} w(\ell)^{\frac{k-1}{k}}\ge w(r)^{\frac{k-1}{k}}.
$$
Now define $H_2=\left(H_1\setminus\{i_1\}\right)\cup\{k+1\}$ and repeat the above argument with $H_2$ replacing $H_1$. We deduce that there exists $i_2\in H_2$ satisfying
$$
\sum_{\ell\in \L\left(T^{(i_2)}\right)} w(\ell)^{\frac{k-1}{k}}\ge w(r)^{\frac{k-1}{k}}.
$$
We may repeat this process inductively $h-k+1$ times, and let $L=\{i_1,i_2,\ldots,i_{h-k+1}\}$.
\end{proof}

\begin{proof}[Proof of Lemma~\ref{lem:sparse-s-2}]
The proof is by induction on the
height of the tree, but we need the following strengthening of the
inductive hypothesis so as to deal with multiple trees. Suppose that
we are given a collection of (disjoint) subadditive weighted trees
with designated children $(T_1,w,\dc)\ldots,(T_\ell,w,\dc)$, each
$T_i$ is rooted at $r_i$ and all of them having the same height,
which is divisible by $h$ (formally we should denote the weighting
of $T_i$ by $w_i$, but since the trees are disjoint, denoting all
the weightings by $w$ will not create any confusion). We will prove
that there exists a subset $C\subseteq \{1,\ldots,\ell\}$ with the
following properties.
\begin{itemize}
\item For every $i\in C$ there is a subtree $T'_i$ of $T_i$ rooted at
$r_i$, and subsets $S_i, R_i\subseteq T'_i$, both containing the
root of $T_i'$, such that for any non-leaf $u\in
T_i'$ we have $\dc(u)\in T_i'$.
\item $R_i=\{v\in T_i':\ h\mid \depth_{T_i}(v)\}$.
\item For any non-leaf vertex $u\in R_i$,
\begin{equation}\label{eq:power-kidsi}
 \sum_{v\in D_{T_i'}(u,R_i)} w(v)^{\left(1-\frac{1}{k}\right)^2} \ge w(u)^{\left(1-\frac{1}{k}\right)^2} .
 \end{equation}
 \item For every $u,v\in R_i$ such that $\depth_{T_i}(v)=\depth_{T_i}(u)+h$ and $v$ is a descendant of $u$, there is one and only one $w\in S_i$ such that $w$ lies on the path joining $u$ and $v$ and $\depth_{T_i}(u)<\depth_{T_i}(w)\le \depth_{T_i}(v)$.
\item For any $u\in T'_i$ such that $D_{T'_i}(u,S_i)\neq\emptyset$, all the vertices of $D_{T_i'}(u,S_i)$ are at the same depth in $(T_i')_u$,
which is an integer between $1$ and $2h$.
\item $T_i'$ is sparsified at the subset $S_i$.
\item The vertices in $\bigcup_{i\in C} D_{T_i'}(r_i,S_i)$ have the
same depth (in their respective tree), regardless of $i$, and
\begin{equation}\label{eq:double power}
\sum_{i\in C} w(r_i)^{\left(1-\frac{1}{k}\right)^2}\ge
\left(\sum_{i=1}^\ell w(r_i)^{1-\frac{1}{k}}\right)^{1-\frac{1}{k}}.
\end{equation}
\end{itemize}
Note that Lemma~\ref{lem:sparse-s-2} is the case $\ell=1$ of this
statement, but it will be beneficial to prove the more general
statement as formulated above.

When the height of all the $T_i$ is $0$ we simply set
$C=\{1,\ldots,\ell\}$ and $T'_i=S_i=R_i=T_i$. Most of the above
conditions hold vacuously in this case (there are no non-leaf
vertices and $D_{T'_i}(u,S_i)=\emptyset$).
Condition~\eqref{eq:double power} follows from subadditivity of the
function $(0,\infty)\ni t\mapsto t^{1-\frac{1}{k}}$.

Assume next that the  leaves of $\{T_i\}_{i=1}^\ell$ are all at
depth $mh$ for some $m\in\mathbb N$. Let $\widehat T_j$ be the
 subgraph of $T_j$ induced on all the vertices of depth at most $h$ in
$T_j$. Note that by~\eqref{eq:explicit w_h^k} the restriction of the
induced weight function $w^k_h$ to $\widehat T_j$ coincides with the
corresponding weight function induced by the weighted tree
$\left(\widehat T_j,w\right)$, and consequently the same can be said
about the designated child map $\dc|_{\widehat T_j}$. Therefore, an
application of Lemma~\ref{lem:holder} to $\left(\widehat
T_j,w,\dc\right)$ yields a subset $L_j\subseteq \{1,\ldots,h\}$ with
$|L_j|= h-k+1$ such that for all $i\in L_j$ we have
\begin{equation}\label{eq:use the holder}
\sum_{\substack{u\in \widehat{T}_j^{(i)}\\ \depth_{T_j}(u)=h}}
w(u)^{1-\frac{1}{k}}\ge w(r_j)^{1-\frac1{k}},
\end{equation}
where $\widehat{T}_j^{(i)}$ is the subtree of $\widehat{T}_j$ that
is obtained by sparsifying the $i$th level as in
Definition~\ref{def:sparsified}.

Let $j_0\in \{1,\ldots,\ell\}$ satisfy
\begin{equation}\label{eq:def j_0}
w(r_{j_0})=\max_{j\in \{1,\ldots,\ell\}}w(r_j). \end{equation}
 For
$j\in \{1,\ldots,\ell\}$ denote $L_j'=L_j\cap L_{j_0}$. Then
\begin{equation}\label{eq:Lj' size}
|L_j'|=|L_j|+|L_{j_0}|-|L_j\cup L_{j_0}|\ge (h-k+1)+(h-k+1)-h=
h-2(k-1).
\end{equation}
If $\ell=1$ let $s_0\in L_{j_0}$ be an arbitrary integer in
$L_{j_0}$. If $\ell\ge 2$ let $s_0\in L_{j_0}$ satisfy
\begin{equation}\label{eq:def s_0}
\sum_{\substack{j\in \{1,\ldots,\ell\}\setminus \{j_0\}\\s_0\in
L_j'}} w(r_j)^{1-\frac1{k}}=\max_{s\in L_{j_0}} \sum_{\substack{j\in
\{1,\ldots,\ell\}\setminus \{j_0\}\\s\in L_j'}}
w(r_j)^{1-\frac1{k}}.
\end{equation}
By averaging we see that
\begin{multline}\label{eq:averaged lower}
\sum_{\substack{j\in \{1,\ldots,\ell\}\setminus \{j_0\}\\s_0\in L_j'}}
w(r_j)^{1-\frac1{k}}\ge \frac{1}{h-k+1}\sum_{s\in
L_{j_0}}\sum_{\substack{j\in \{1,\ldots,\ell\}\setminus \{j_0\}\\s\in
L_j'}} w(r_j)^{1-\frac1{k}}\\=\frac{1}{h-k+1} \sum_{j\in
\{1,\ldots,\ell\}\setminus
\{j_0\}}|L_j'|w(r_j)^{1-\frac1{k}}\stackrel{\eqref{eq:Lj'
size}}{\ge} \frac{h-2(k-1)}{h-(k-1)}\sum_{j\in
\{1,\ldots,\ell\}\setminus \{j_0\}}w(r_j)^{1-\frac1{k}}.
\end{multline}

Now define
\begin{equation}\label{eq:def C}
C=\left\{j\in \{1,\ldots,\ell\}: s_0\in L_j\right\}.
\end{equation}
We know that $j_0\in C$, since by construction  $s_0\in  L_{j_0}$.
By~\eqref{eq:Lj' size} we have $L'_j\neq\emptyset$ for every
$j\in\{1,\ldots,\ell\}$. Therefore~\eqref{eq:def s_0} implies  that
if $\ell\ge 2 $ then $|C|\ge 2$. Since~\eqref{eq:double power} is
trivial when $\ell=1$, we will now prove~\eqref{eq:double power}
assuming $\ell\ge 2$. By the choice of $j_0$ in~\eqref{eq:def j_0}
we know that for all $j\in \{1,\ldots,\ell\}\setminus \{j_0\}$ we
have
\begin{equation}\label{eq:half}
w\left(r_j\right)^{1-\frac1{k}}\le \frac12\sum_{i=1}^\ell
w(r_i)^{1-\frac{1}{k}}.
\end{equation}
This implies that for all $j\in \{1,\ldots,\ell\}\setminus \{j_0\}$,
\begin{equation}\label{eq:power of 2 appears}
w(r_j)^{\left(1-\frac{1}{k}\right)^2}\ge
\frac{h-(k-1)}{h-2(k-1)}\cdot
\frac{w(r_j)^{1-\frac{1}{k}}}{\left(\sum_{i=1}^\ell
w(r_i)^{1-\frac{1}{k}}\right)^{\frac{1}{k}}}.
\end{equation}
To check~\eqref{eq:power of 2 appears} note that it is equivalent to
the inequality
$$
\frac{h-(k-1)}{h-2(k-1)}\left(\frac{w(r_j)^{1-\frac{1}{k}}}{\sum_{i=1}^\ell
w(r_i)^{1-\frac{1}{k}}}\right)^{\frac{1}{k}}\le 1.
$$
By~\eqref{eq:half} and the fact that $h\ge 2k^2$, it suffices to
show that $(2k^2-k+1)/(2k^2-2k+2)\le 2^{1/k}$. Since $2^{1/k}\ge
1+1/(2k)$ it suffices to check that $2k(2k^2-k+1)\le
(2k+1)(2k^2-2k+2)$, which is immediate to verify.

Having established~\eqref{eq:power of 2 appears}, we proceed as
follows.
\begin{eqnarray*}
\sum_{j\in C}
w(r_j)^{\left(1-\frac{1}{k}\right)^2}&=&\frac{w(r_{j_0})^{1-\frac{1}{k}}}
{w(r_{j_0})^{\left(1-\frac{1}{k}\right)\frac{1}{k}}}+\sum_{j\in
C\setminus\{j_0\}}w(r_j)^{\left(1-\frac{1}{k}\right)^2}\\&\stackrel{\eqref{eq:power
of 2 appears}}{\ge}&
\frac{w(r_{j_0})^{1-\frac{1}{k}}}{\left(\sum_{i=1}^\ell
w(r_i)^{1-\frac{1}{k}}\right)^{\frac{1}{k}}}+\frac{h-(k-1)}{h-2(k-1)}\cdot\frac{\sum_{j\in
C\setminus \{j_0\}}w(r_j)^{1-\frac{1}{k}}}{\left(\sum_{i=1}^\ell
w(r_i)^{1-\frac{1}{k}}\right)^{\frac{1}{k}}}\\
&\stackrel{\eqref{eq:averaged lower}\wedge\eqref{eq:def C}}{\ge}&
\frac{w(r_{j_0})^{1-\frac{1}{k}}}{\left(\sum_{i=1}^\ell
w(r_i)^{1-\frac{1}{k}}\right)^{\frac{1}{k}}}+\frac{\sum_{j\in
\{1,\ldots,\ell\}\setminus
\{j_0\}}w(r_j)^{1-\frac1{k}}}{\left(\sum_{i=1}^\ell
w(r_i)^{1-\frac{1}{k}}\right)^{\frac{1}{k}}}\\
&=& \left(\sum_{i=1}^\ell
w(r_i)^{1-\frac{1}{k}}\right)^{1-\frac{1}{k}},
\end{eqnarray*}
completing the proof of~\eqref{eq:double power}.

We can now complete the proof of Lemma~\ref{lem:sparse-s-2} by
applying the inductive hypothesis. For every $j\in C$ let
$u_1^j,u_2^j,\ldots,u_{\ell_j}^j$ be the leaves of the tree
$\widehat T_j^{(s_0)}$, i.e., the subtree of $\hat T_j$ that was
sparsified at level $s_0$. Consider the subtrees of $T_j$ that are
rooted at $u_1^j,u_2^j,\ldots,u_{\ell_j}^j$, i.e.,
$(T_j)_{u_1^j},(T_j)_{u_2^j},\ldots,(T_j)_{u_{\ell_j}^j}$. By the
inductive hypothesis applied to these trees there exists a subset
$C_j\subseteq \{1,\ldots,\ell_j\}$ such that for each $i\in C_j$
there is a subtree $(T_j)_{u_i^j}'$ of $(T_j)_{u_i^j}$ and subsets
$S_{ji},R_{ji}\subseteq (T_j)_{u_i^j}'$ that satisfy the inductive
hypotheses.

Denote for $j\in C$,
$$
T_j'=\left(\bigcup_{i\in C_j} (T_j)_{u_i^j}'\right)\bigcup \left(\bigcup_{i\in C_j}\left\{u\in \widehat T_j^{(s_0)}:\ u\ \mathrm{ancestor\ of\ } u_i^j\right\}\right).
$$
Thus $T_j'$ is obtained by taking the subtree of $\widehat T_j^{(s_0)}$ whose leaves are $\left\{u_i^j\right\}_{i\in C_j}$, and replacing every leaf $u_i^j$ by the tree $(T_j)_{u_i^j}'$. We also define
$
R_j=\left(\bigcup_{i=1}^{\ell_j} R_{ji}\right)\bigcup \{r_j\},$
and
$$ S_j=
\left(\bigcup_{i=1}^{\ell_j} \left(S_{ji}\setminus\left\{u_i^j\right\}\right)\right)\bigcup\left\{u\in\widehat T_j^{(s_0)}:\ \depth_{T_j}(u)=s_0 \right\}\bigcup\{r_j\}.
$$
Note by the definition of $\widehat T_j^{(s_0)}$  every $u\in \widehat T_j^{(s_0)}$ with $\depth_{T_j}(u)=s_0$ has no siblings in $\widehat T_j^{(s_0)}$.

All the desired properties of $T_j',R_j,S_j$ follow immediately for the construction; only~\eqref{eq:power-kidsi} when $u=r_{j}$ requires justification as follows.
\begin{equation*}
\sum_{v\in D_{T_j'}(r_j,R_j)}w(v)^{\left(1-\frac{1}{k}\right)^2}=\sum_{i\in C_j} w\left(u_i^j\right)^{\left(1-\frac{1}{k}\right)^2}\stackrel{\eqref{eq:double power}}{\ge}\left(\sum_{i=1}^{\ell_j} w\left(u_i^j\right)^{1-\frac1{k}}\right)^{1-\frac1{k}}\stackrel{\eqref{eq:use the holder}}{\ge} w(r_j)^{\left(1-\frac{1}{k}\right)^2}.\qedhere
\end{equation*}
\end{proof}

\section{Proof of Theorem~\ref{lem:theta(D)}}
\label{sec:NT}

Here we prove Theorem~\ref{lem:theta(D)}, which is the last missing ingredient of the proof of
Theorem~\ref{thm:measure}. Theorem~\ref{lem:theta(D)} is a
weighted version of the results
of~\cite{MN07,NT-fragmentations}. Theorem~\ref{thm:two weight} below
follows from a slight modification of the argument
in~\cite{NT-fragmentations}, though it is not stated there
explicitly. We will therefore explain how~\cite{NT-fragmentations}
can be modified to deduce this statement. Alternatively, a similar
statement (with worse distortion bound) can be obtained by
natural modifications of the argument in~\cite{MN07}.

\begin{theorem}\label{thm:two weight}
Let $(X,d)$ be a finite metric space and $w_1,w_2:X\to [0,\infty)$
two nonnegative weight functions. Then for every $\e\in (0,1)$ there
exists a subset $S\subseteq X$ that embeds into an ultrametric space
with distortion
\begin{equation}\label{eq:dist bound}
D=\frac{2}{\e(1-\e)^\frac{1-\e}{\e}},
\end{equation}
and satisfying
\begin{equation}\label{eq:two weight}
\left(\sum_{x\in S} w_1(x)\right)\left(\sum_{x\in X}
w_2(x)\right)^\e\ge \sum_{x\in X} w_1(x)w_2(x)^\e.
\end{equation}
\end{theorem}

Theorem~\ref{thm:two weight} implies the finite nonlinear Dvoretzky theorem (Theorem~\ref{thm:1/eps}) in the special case $w_1=w_2=1$.
Theorem~\ref{lem:theta(D)} follows by taking $w_1=w^{1-\e}$ and $w_2=w$.
In this case conclusion~\eqref{eq:two weight} becomes
\begin{equation}\label{eq:BLMN weighted}
\sum_{x\in S} w(x)^{1-\e}\ge \left(\sum_{x\in X}w(x)\right)^{1-\e}.
\end{equation}
This type of requirement was studied in~\cite{BLMN05} under the name of ``the weighted metric Ramsey problem", where is was shown that there always exists $S\subseteq X$ satisfying~\eqref{eq:BLMN weighted} that embeds into an ultrametric space with distortion $O\left(\e^{-1}\log(2/\e)\right)$.

\begin{proof}[Proof of Theorem~\ref{thm:two weight}]
The beginning of the argument is most natural to state in the
context of general metric measure spaces $(X,d,\mu)$. So, assume
that $(X,d,\mu)$ is a metric measure space; we will later specialize
the discussion to the case of finite spaces.

Let $f:X\to [0,\infty)$ be a nonnegative Borel measurable function.
Lemma~2.1 of~\cite{NT-fragmentations} states that for every compact
$S\subseteq X$ and every $R>r>0$ there exists a compact subset
$T\subseteq S$ satisfying
\begin{equation}\label{eq:NT for iteration}
\int_T\frac{\mu(B(x,R))}{\mu(B(x,r))}f(x)d\mu(x)\ge
\int_{S}fd\mu,
\end{equation}
such that $T$ can be partitioned as $T = \bigcup_{n=1}^\infty T_n$,
where each (possibly empty) $T_n$ is compact and contained in a ball
of radius $r$, and any two non-empty $T_n, T_m$ are separated by a
distance of at least $R-r$.

Fix a nonnegative Borel measurable $w\in L_1(\mu)$. Iterate the
above statement as follows; the same iteration is carried out for
the special case $w=1$ in Lemma~2.2 of~\cite{NT-fragmentations}.
Assume that we are given a non-increasing sequence of positive
numbers $R=r_0\ge r_1\ge r_2\ge \cdots>0$ converging to zero. Assume
also that $\diam(X)\le 2R$. For $n\in \N$ define $f_n:X\to
[0,\infty)$ by
\begin{equation}\label{eq:def f_n}
f_n(x)=\left(\prod_{m =n}^\infty \frac{\mu(B(x,r_m))}
{\mu\left(B\left(x,r_m +
\frac{2r_{m-1}}{D}\right)\right)}\right)w(x),
\end{equation}
Where $D$ be given
by~\eqref{eq:dist bound}. Note that $0\le f_n\le w$ for all $n\in \N$.  Assume that we already defined a compact
subset $S_{n-1}\subseteq X$. An application of Lemma~2.1
of~\cite{NT-fragmentations}, with radii $r_n+\frac{2r_{n-1}}{D}>r_n$
and weight function $f_n$, yields a compact subset $S_n\subseteq
S_{n-1}$ satisfying
$$
\int_{S_n}f_{n+1}d\mu=
\int_{S_n}\frac{\mu\left(B\left(x,r_n+\frac{2r_{n-1}}{D}\right)\right)}{\mu(B(x,r_n))}f_n(x)d\mu(x)
\stackrel{\eqref{eq:NT for iteration}}{\ge} \int_{S_{n-1}} f_nd\mu.
$$
Hence for all $n\in \N$ we have,
$$
\int_{S_n}wd\mu\ge \int_Xf_1d\mu.
$$
Consider the compact subset $S=\bigcap_{n=1}^\infty S_n$. By the
dominated convergence theorem,
\begin{equation}\label{eq:int S}
\int_{S}wd\mu\ge \int_Xf_1d\mu.
\end{equation}
In~\cite[Lem.~2.2]{NT-fragmentations} it is shown that $S$
embeds with distortion $D$ into an ultrametric space.

Assume now that the radii $1=r_0\ge r_1\ge r_2\ge \cdots>0$ are random variables satisfying $\lim_{n\to \infty} r_n=0$ and for every real number $r>0$
\begin{equation}\label{eq:addmissible}
\Pr\left[r_n<r\le r_n+\frac{2r_{n-1}}{D}\right]\le \e.
\end{equation}
For the existence of such random variables, as well as the optimality for this purpose of the choice of $D$ in~\eqref{eq:dist bound}, see~\cite[Thm.~1.5]{NT-fragmentations}. Specializing to the case of a finite metric measure space $(X,d,\mu)$ of diameter at most $2$,  apply~\eqref{eq:int S} when $w(x)=w_1(x)/\mu(\{x\})$ and $w_2(x)=\mu(\{x\})$, and the radii are the random radii chosen above. By taking expectation of the resulting (random) inequality and using Jensen's inequality we arrive at the following estimate.
\begin{equation}\label{eq:jensen}
\E\left[\sum_{x\in S} w_1(x)\right]\stackrel{\eqref{eq:def f_n}\wedge\eqref{eq:int S}}{\ge} \sum_{x\in X} w_1(x)\exp\left(\E\left[\sum_{n=1}^\infty\log\left(\frac{\mu(B(x,r_m))}
{\mu\left(B\left(x,r_m +
\frac{2r_{m-1}}{D}\right)\right)}\right)\right]\right).
\end{equation}

For every $x\in X$ let $0=t_1(x)<t_2(x)<\ldots<t_{k(x)}(x)$ be the radii at which $\mu(B(x,t))$ jumps, i.e.,
 $\mu(\{x\})=\mu(B(x,t_1(x)))<\mu(B(x,t_2(x)))<\ldots<\mu(B(x,t_{k(x)}(x)))=\mu(X)$, and
$B(x,t)=B(x,t_j(x))$ if $t_j(x)\le t<t_{j+1}(x)$ (where we use the
convention $t_{k(x)+1}(x)=\infty$). Then we have the following straightforward identity (see equation (15) in~\cite{NT-fragmentations}), which holds for every $x\in X$.
\begin{multline}\label{eq:fubini identity}
\E\left[\sum_{n=1}^\infty\log\left(\frac{\mu(B(x,r_m))}
{\mu\left(B\left(x,r_m +
\frac{2r_{m-1}}{D}\right)\right)}\right)\right]\\=-\sum_{j=2}^{k(x)}\left(\sum_{n=1}^\infty\Pr\left[r_n < t_j(x) \leq r_n + \frac{2r_{n-1}}{D}\right]\right)\log\left(\frac{\mu\left(B(x,t_j(x))\right)}{\mu\left(B(x,t_{j-1}(x))\right)}\right).
\end{multline}
Hence,
\begin{multline*}\label{eq:expectation}
\E\left[\sum_{x\in S} w_1(x)\right]\stackrel{\eqref{eq:addmissible}\wedge\eqref{eq:jensen}\wedge\eqref{eq:fubini identity}}{\ge} \sum_{x\in X} w_1(x) \exp\left(-\e\sum_{j=2}^{k(x)}\log\left(\frac{\mu\left(B(x,t_j(x))\right)}
{\mu\left(B(x,t_{j-1}(x))\right)}\right)\right)\\
=\sum_{x\in X} w_1(x) \left(\frac{\mu(\{x\})}{\mu(X)}\right)^{\e}=\frac{\sum_{x\in X} w_1(x)w_2(x)^\e}{\left(\sum_{y\in X}w_2(y)\right)^\e}.
\end{multline*}
We have shown that the required estimate~\eqref{eq:two weight} holds in expectation for our random subset $S\subseteq X$, completing the proof of Theorem~\ref{thm:two weight}.
\end{proof}


\section{Impossibility results} \label{sec:ub}

The purpose of this section is to prove the second part of Theorem~\ref{thm:hausdorff into} and Theorem~\ref{thm:<2}. In both cases the goal is to construct a metric space having the property that all its ``almost Euclidean" subsets have small Hausdorff dimension. We will do so by gluing together the finite examples from~\cite{BLMN05}: in the high distortion regime corresponding to Theorem~\ref{thm:hausdorff into} these building blocks are expander graphs, and in the low distortion regime corresponding to Theorem~\ref{thm:<2} these building blocks  arise from dense random graphs. The gluing procedure, which is an infinitary variant of the ``metric composition" method from~\cite{BLMN05}, starts with a sequence of finite metric spaces and joins them in a tree-like fashion. The details of the construction are contained in Section~\ref{sec:trees of metrics} below, and the specializations to prove Theorem~\ref{thm:hausdorff into} and Theorem~\ref{thm:<2} are described in Section~\ref{sec:expander} and Section~\ref{sec:random graph}, respectively.

\subsection{Trees of metric spaces}\label{sec:trees of metrics}

Fix $\{n_k\}_{k=0}^\infty\subseteq \N$ with $n_0=1$ and $n_k>1$ for $k\ge 1$. Fix also $\{\delta_k\}_{k=1}^\infty\subseteq (0,\infty)$. Assume that for each $k\in \N$ we are given a metric $d_k$ on $\{1,\ldots,n_k\}$ with
\begin{equation}\label{eq:normalizations}
\diam_{d_k}\left(\{1,\ldots,n_k\}\right)=1\quad\mathrm{and}\quad \min_{i,j\in \{1,\ldots,n_k\}} d_k(i,j)=\delta_k.
\end{equation}
For distinct $x=(x_k)_{k=1}^\infty,y=(y_k)_{k=1}^\infty\in \prod_{k=1}^\infty\{1,\ldots,n_k\}$ let $k(x,y)$ be the smallest $k\in \N$ such that $x_k\neq y_k$. For $\alpha\in (0,\infty)$ define
\begin{equation}\label{eq:def rho alpha}
\rho_\alpha(x,y)=\frac{d_{k(x,y)}\left(x_{k(x,y)},y_{k(x,y)}\right)}{\prod_{i=0}^{k(x,y)-1}n_i^{1/\alpha}}.
\end{equation}
Also, set $\rho_\alpha(x,x)=\rho_\alpha(y,y)=0$.

\begin{remark}\label{rem:visualize}
One can visualize the above construction as follows. Let $T$ be the infinite rooted tree such for $i\ge 0$ each vertex at depth $i$ in $T$ has exactly  $n_{i+1}$ children. Then $\prod_{k=1}^\infty\{1,\ldots,n_k\}$ can be identified with the set of all infinite branches of $T$. With this identification, the distance $\rho_\alpha$ has the following meaning: given two infinite branches in $T$, find the vertex $v\in T$ at which they split (i.e., their deepest common vertex). Say that the depth of $v$ is $i-1$. The metric $d_i$ induces a metric space structure on the $n_{i}$ children of $v$, and the distance between the two given branches is a multiple of the distance between the two children of $v$ that belong to these branches.
\end{remark}

\begin{lemma}\label{lem:is compact}
$\left(\prod_{k=1}^\infty\{1,\ldots,n_k\},\rho_\alpha\right)$ is a compact metric space provided that
\begin{equation}\label{eq:delta_k condition}
\forall k\in \N,\quad \delta_k>\frac{1}{n_k^{1/\alpha}}.
\end{equation}
\end{lemma}

\begin{proof}
Take $x,y,z\in \prod_{k=1}^\infty\{1,\ldots,n_k\}$. If $k(x,y)=k(y,z)=k(x,z)=k$
 then because $d_k$ satisfies the triangle inequality, $\rho_\alpha(x,z)\le \rho_\alpha(x,y)+\rho_\alpha(y,z)$. If $k(x,y)>k(x,z)$ then necessarily $k(x,z)=k(y,z)=k$ and $x_k=y_k$. Hence $\rho_\alpha(x,z)=\rho_\alpha(x,y)\le \rho_\alpha(x,y)+\rho_\alpha(y,z)$. The remaining case $k(x,z)>k(x,y)$ is dealt with as follows.
\begin{equation*}
\rho_\alpha(x,z)=\frac{d_{k(x,z)}\left(x_{k(x,z)},z_{k(x,z)}\right)}{\prod_{i=0}^{k(x,z)-1}n_i^{1/\alpha}}
\stackrel{\eqref{eq:normalizations}}{\le}\frac{1}{\prod_{i=0}^{k(x,z)-1}n_i^{1/\alpha}}\stackrel{\eqref{eq:delta_k condition}}{\le} \frac{\delta_{k(x,y)}}{\prod_{i=0}^{k(x,y)-1}n_i^{1/\alpha}}\stackrel{\eqref{eq:normalizations}}{\le} \rho_\alpha(x,y).
\end{equation*}
This proves the triangle inequality. Compactness follows from Tychonoff's theorem since $\rho_\alpha$ induces the product topology on $\prod_{k=1}^\infty\{1,\ldots,n_k\}$. 
\end{proof}

\begin{lemma}\label{lem:compute hausdorff}
Assume that in addition to~\eqref{eq:delta_k condition} we have
\begin{equation}\label{eq:limit}
\lim_{k\to \infty}\frac{\log(1/\delta_k)}{\sum_{i=1}^{k-1}\log n_i}=0.
\end{equation}
Then
$$
\dim_H\left(\prod_{k=1}^\infty\{1,\ldots,n_k\},\rho_\alpha\right)=\alpha.
$$
\end{lemma}

\begin{proof}
Define $\Delta=\dim_H\left(\prod_{k=1}^\infty\{1,\ldots,n_k\},\rho_\alpha\right)$. The fact that $\Delta\le \alpha$ is simple. Indeed,  for $k\in \N$ consider the sets $\{B^k_x\}_{x\in \prod_{i=1}^k\{1,\ldots,n_i\}}$ given by
\begin{equation}\label{eq:def B_x}
B_x^k=\left(\prod_{i=1}^k \{x_i\}\right)\times\left(\prod_{i=k+1}^\infty\{1,\ldots,n_i\}\right).
\end{equation}
Then $\diam_{\rho_\alpha}(B_x^k)=\prod_{i=1}^k n_i^{-1/\alpha}$ and $\{B^k_x\}_{x\in \prod_{i=1}^k\{1,\ldots,n_i\}}$ cover $\prod_{i=1}^\infty\{1,\ldots,n_i\}$. Hence for $\beta>\alpha$ the $\beta$-Hausdorff content of $\left(\prod_{i=1}^\infty\{1,\ldots,n_i\},\rho_\alpha\right)$ can be estimated as follows.
$$
\mathcal{H}_\infty^\beta\left(\prod_{i=1}^\infty\{1,\ldots,n_i\},\rho_\alpha\right)\le \inf_{k\in \N} \sum_{x\in \prod_{i=1}^k\{1,\ldots,n_i\}}\frac{1}{\prod_{i=0}^kn_i^{\beta/\alpha}}=\inf_{k\in \N} \frac{1}{\prod_{i=0}^kn_i^{(\beta-\alpha)/\alpha}}=0.
$$
Thus $\Delta\le \alpha$.

We now pass to the proof of $\Delta\ge \alpha$. We first prove the following preliminary statement. Assume that $x^1,\ldots,x^m\in  \prod_{i=1}^\infty\{1,\ldots,n_i\}$ and $k_1,\ldots,k_m\in \N\cup \{0\}$ are such that $\{B_{x^j}^{k_j}\}_{j=1}^m$ cover $\prod_{i=1}^\infty\{1,\ldots,n_i\}$, where $B_x^k$ is given in~\eqref{eq:def B_x} and we use the convention $B_x^0=\prod_{i=1}^\infty\{1,\ldots,n_i\}$. We claim that this implies that
\begin{equation}\label{eq:prem}
\sum_{j=1}^m \frac{1}{\prod_{i=0}^{k_j}n_i}\ge 1.
\end{equation}
The proof is by induction on $m$. If $m=1$ then  $k_1=0$ and~\eqref{eq:prem} follows. Assume that $m\ge 2$, no subset of $\{B_{x^j}^{k_j}\}_{j=1}^m$ covers $\prod_{i=1}^\infty\{1,\ldots,n_i\}$, and that $k_1\le k_2\le\cdots\le k_m$. For every $y\in \{1,\ldots,n_{k_m}\}$ let $x^{k_m}(y)\in \prod_{i=1}^\infty\{1,\ldots,n_i\}$ have $y$ in the $k_m$'th coordinate, and coincide with $x^{k_m}$ in all other coordinates. The sets $\{B_{x^{k_m}(y)}^{k_m}\}_{y\in \{1,\ldots,n_{k_m}\}}$ are pairwise disjoint, and are either contained in or disjoint from $B_{x^j}^{k_j}$ for each $j\in \{1,\ldots,m\}$. Hence, by the minimality of the cover  $\{B_{x^j}^{k_j}\}_{j=1}^m$ we have  $k_m=k_{m-1}=\cdots =k_{m-n_{k_m}+1}$ and $\{B_{x^j}^{k_j}\}_{j=m-n_{k_m}+1}^{k_m}=\{B_{x^{k_m}(y)}^{k_m}\}_{y\in \{1,\ldots,n_{k_m}\}}$. This implies that
\begin{equation*}\label{remove last}
\sum_{j=1}^m \frac{1}{\prod_{i=0}^{k_j}n_i}=\sum_{j=1}^{m-n_{k_m}} \frac{1}{\prod_{i=0}^{k_j}n_i}+\sum_{y\in \{1,\ldots,n_{k_m}\}}\frac{1}{\prod_{i=0}^{k_{m}}n_i}
=\sum_{j=1}^{m-n_{k_m}} \frac{1}{\prod_{i=0}^{k_j}n_i}+\frac{1}{\prod_{i=0}^{k_{m}-1}n_i}.
\end{equation*}
The induction hypothesis applied to $\{B_{x^j}^{k_j}\}_{j=1}^{m-n_{k_m}}\cup\{B_{x^{m}}^{k_m-1}\}$ concludes the proof of~\eqref{eq:prem}.

Fix $\beta\in (0,\alpha)$. Due to~\eqref{eq:limit} there exists $C\in (0,\infty)$ such that for all $k\in \N$,
\begin{equation}\label{eq:use limit delta}
\frac{1}{\delta_k}\le  C\prod_{i=1}^{k-1} n_i^{(\alpha-\beta)/\alpha^2}.
\end{equation}
Let $\{B_{\rho_\alpha}(x^j,r_j)\}_{j\in J}$ be a family of balls that covers $\prod_{i=1}^\infty\{1,\ldots,n_i\}$. We will show that
\begin{equation}\label{eq:hausdorff lower}
\sum_{j\in J} r_j^\beta\ge \frac{1}{C^{\alpha}}.
\end{equation}
This would mean that $\mathcal{H}_\infty^\beta\left(\prod_{i=1}^\infty\{1,\ldots,n_i\},\rho_\alpha\right)>0$ for all $\beta\in (0,\alpha)$, proving that $\Delta\ge \alpha$. By compactness of $\left(\prod_{i=1}^\infty\{1,\ldots,n_i\},\rho_\alpha\right)$ it suffices to prove~\eqref{eq:hausdorff lower} when $J$ is finite. For every $j\in J$ choose $k_j\in \N$ such that $\prod_{i=0}^{k_j}n_i^{-1/\alpha}<r_j\le \prod_{i=0}^{k_j-1}n_i^{-1/\alpha}$. Define
\begin{equation}\label{eq:def r_j^*}
r^*_j=\left\{\begin{array}{ll} \prod_{i=1}^{k_j-1}n_i^{-1/\alpha}& \mathrm{if}\   \delta_{k_j}\prod_{i=1}^{k_j-1}n_i^{-1/\alpha}\le r_j\le\prod_{i=1}^{k_j-1}n_i^{-1/\alpha},\\
\prod_{i=1}^{k_j}n_i^{-1/\alpha}&\mathrm{if}\ \prod_{i=1}^{k_j}n_i^{-1/\alpha}<r_j<\delta_{k_j}\prod_{i=1}^{k_j-1}n_i^{-1/\alpha}.\end{array}\right.
\end{equation}
If $\delta_{k_j}\prod_{i=1}^{k_j-1}n_i^{-1/\alpha}\le r_j\le\prod_{i=1}^{k_j-1}n_i^{-1/\alpha}$ then $r_j^*\ge r_j$, hence $B_{\rho_\alpha}(x^j,r_j)\subseteq B_{\rho_\alpha}(x^j,r^*_j)=B_{x^j}^{k_j-1}$. Also, observe that $\rho_\alpha$ does not take values in the interval $\left(\prod_{i=1}^{k_j}n_i^{-1/\alpha},\delta_{k_j}\prod_{i=1}^{k_j-1}n_i^{-1/\alpha}\right)$. This implies that $B_{\rho_\alpha}(x^j,r_j)= B_{\rho_\alpha}(x^j,r^*_j)=B_{x^j}^{k_j}$ when $\prod_{i=1}^{k_j}n_i^{-1/\alpha}<r_j<\delta_{k_j}\prod_{i=1}^{k_j-1}n_i^{-1/\alpha}$. We deduce that the balls $B_{\rho_\alpha}(x^j,r^*_j)$ cover $\prod_{i=1}^\infty\{1,\ldots,n_i\}$, and they are all sets of the form $B_x^k$. It therefore follows from~\eqref{eq:prem} that
\begin{equation}\label{eq:hausdorff done}
1\le \sum_{j\in J} \left(r_j^*\right)^\alpha\stackrel{\eqref{eq:def r_j^*}}{\le} \sum_{j\in J} \left(\frac{r_j}{\delta_{k_j}}\right)^\alpha\stackrel{\eqref{eq:use limit delta}}{\le} C^\alpha \sum_{j\in J} r_j^\alpha\prod_{i=0}^{k_j-1}n_i^{(\alpha-\beta)/\alpha}\le C^\alpha\sum_{j\in J} r_j^\beta,
\end{equation}
where in the last inequality of~\eqref{eq:hausdorff done} we used the fact that $r_j\le \prod_{i=0}^{k_j-1}n_i^{-1/\alpha}$.
\end{proof}

\begin{remark}\label{rem:frostman}
A less direct way to prove the bound $\Delta\ge \alpha$ in Lemma~\ref{lem:compute hausdorff} is to define $\mu(B_x^k)=\prod_{i=0}^k n_i^{-1}$ and to argue that the Carath\'eodory extension theorem applies here and yields an extension of $\mu$ to a Borel measure on $\prod_{i=1}^\infty\{1,\ldots,n_i\}$. One can then show analogously to~\eqref{eq:hausdorff done} that this measure is a $\beta$-Frostman measure for $\left(\prod_{i=1}^\infty\{1,\ldots,n_i\},\rho_\alpha\right)$.
\end{remark}

In what follows we say that a property  $\P$ of metric spaces is a metric property if whenever $(X,d_X)\in P$ and $(Y,d_Y)$ is isometric to $(X,d_X)$ then also $(Y,d_Y)\in \P$. We say that $\P$ is hereditary if whenever $(X,d)\in \P$ and $Y\subseteq X$ then also $(Y,d)\in \P$. Finally, we say that $\P$ is dilation-invariant if whenever $(X,d)\in \P$ and $\lambda\in (0,\infty)$ also $(X,\lambda d)\in \P$.

\begin{theorem} \label{lem:ub-second-lem}
Fix $\alpha>0$. Fix also $\{n_k\}_{k=0}^\infty\subseteq \N$ with $n_0=1$ and $n_k>1$ for $k\ge 1$, a sequence $\{\delta_k\}_{k=1}^\infty\subseteq (0,\infty)$, and for each $k\in \N$ a metric $d_k$ on $\{1,\ldots,n_k\}$. Assume that~\eqref{eq:normalizations}, \eqref{eq:delta_k condition} and~\eqref{eq:limit} hold true. Then there exists a metric space $(Y,\rho)$ with $\dim_H(Y,\rho)=\alpha$ that satisfies the following property. Let $\{\P_k\}_{i=1}^\infty$ is a non-decreasing (with respect to inclusion) sequence of hereditary dilation-invariant metric properties and for every $k\in \N$ let $m_k$ be the cardinality of the largest subset $S$ of $\{1,\ldots,n_k\}$ such that $(S,d_k)$ has the property $\P_k$. Then every $Z\subseteq Y$ that has the property $\bigcup_{k=1}^\infty\P_k$ satisfies
$$
\dim_H(Z,\rho)\le \limsup_{k\to \infty} \frac{\alpha\log m_k}{\log n_k}.
$$
\end{theorem}
\begin{proof} Take $(Y,\rho)=\left(\prod_{k=1}^\infty \{1,\ldots,n_k\},\rho_\alpha\right)$, where $\rho_\alpha$ is given in~\eqref{eq:def rho alpha}. By Lemma~\ref{lem:compute hausdorff} and Lemma~\ref{lem:is compact} we know that $(Y,\rho)$ is a compact metric space of Hausdorff dimension $\alpha$.

Assume that $Z\subseteq Y$ and $(Z,\rho)\in \P_K$ for some $K\in \N$. Since $\{\P_k\}_{i=1}^\infty$ are non-decreasing properties, we know that $(Z,\rho)\in \P_k$ for all $k\ge K$. For every $x\in Y$ and $k\in \N$ denote
$$
S_x^k=\left\{j\in \{1,\ldots,n_{k}\}:\ Z\cap\left(\left(\prod_{i=1}^{k-1}\{x_i\}\right)\times \{j\}\times \left(\prod_{i=k+1}^\infty\{1,\ldots,n_i\}\right)\right)\neq\emptyset \right\}.
$$
If $j\in S_x^k$ choose $x^k(j)\in Z$ whose first $k-1$ coordinates coincide with the corresponding coordinates of $x$, and whose $k$'th coordinate equals $j$. Then $\{x^k(j)\}_{j\in S_x^k}$ is a subset of $Z$ whose metric is isometric to a dilation of the metric $d_k$ on $S_x^k$. Since $\P_k$ is a hereditary dilation-invariant metric property, it follows that $S_x^k\in \P_k$ for all $k\ge K$. Hence $|S_x^k|\le m_k$. Let $Z_k$ be the projection of the set $Z$ onto the first $k$ coordinates, i.e., the set of all $x\in \prod_{i=1}^k\{1,\ldots,n_i\}$ such that $B_x^k\cap Z\neq\emptyset$, where $B_x^k$ is given in~\eqref{eq:def B_x}. Then it follows by induction that for every $k\ge K$ we have $|Z_k|\le \prod_{i=0}^{K-1}n_i\cdot\prod_{i=K}^km_i$. Denote $\gamma=\limsup_{k\to \infty} \frac{\alpha\log m_k}{\log n_k}$. If $\beta>\gamma$ then there exists $K'\ge K$ such that for every $k\ge K'$ we have $m_k\le n_k^{(\beta+\gamma)/(2\alpha)}$. Since the sets $\{B_x^k\}_{x\in Z_k}$ cover $Z$ and have $\rho$-diameter $\prod_{i=0}^kn_i^{-1/\alpha}$,
\begin{multline*}
\mathcal{H}_\infty^\beta(Z,\rho)\le \inf_{k\ge K'} \sum_{x\in Z_k} \frac{1}{\prod_{i=0}^k n_i^{\beta/\alpha}}\le \inf_{k\ge K'} \frac{\prod_{i=0}^{K-1}n_i\cdot\prod_{i=K}^km_i}{\prod_{i=0}^k n_i^{\beta/\alpha}}\\\le
\inf_{k\ge K'} \frac{\prod_{i=0}^{K'-1}n_i\cdot\prod_{i=K'}^kn_i^{(\beta+\gamma)/(2\alpha)}}{\prod_{i=0}^k n_i^{\beta/\alpha}}\le
\inf_{k\ge K'} \frac{\prod_{i=0}^{K'-1}n_i}{\prod_{i=K'}^kn_i^{(\beta-\gamma)/(2\alpha)}}=0.
\end{multline*}
Hence $\dim_H(Z,\rho)\le \gamma$.
\end{proof}

\begin{corollary} \label{cor:ub-first-lem}
Fix an integer $n\ge 2$. Let $(X,d)$ be an $n$-point metric space, and assume that $\Psi\in (0,\infty)$ satisfies
$$
\Psi\cdot\min_{\substack{x,y\in X\\x\neq y}}d(x,y) > \diam(X).
$$
Then there exists a compact metric space $(Y,\rho)$ with $\dim_H(Y,\rho)=\log_\Psi n$ that has the following property. Fix $m\in \{1,\ldots,n-1\}$ and assume that $\P$ is a hereditary dilation-invariant metric property such that the largest subset of $X$ having the property $\P$ is of size $m$. Then $\dim_H(Z,\rho)\le \log_\Psi m$ for every $Z\subseteq Y$ with property $\P$.
\end{corollary}
\begin{proof}
By rescaling assume that $\diam(X)=1$. Now apply Theorem~\ref{lem:ub-second-lem} with  $X_i=X$, $n_i=n$, $\P_i=\P$ and $\alpha =\log_\Psi n$.
\end{proof}
\subsection{Expander fractals}\label{sec:expander} It is shown in~\cite{BLMN05} that there is $c\in (0,\infty)$ such that for any $n\in \N$ there exists an $n$-point metric space $X_n$ such that for every $\e\in(0,1)$ all the subsets of $X_n$ of cardinality greater than $n^{1-\e}$ incur distortion greater than $c/\e$ in any embedding into Hilbert space. In fact, the spaces $X_n$ are the shortest-path metrics on expander graphs, implying that $\diam X_n\le C\log n$ for some $C\in (0,\infty)$ and all $n\in \N$ (see~\cite{Chu89}). We will apply Corollary~\ref{cor:ub-first-lem} to these spaces, thus obtaining compact metric spaces that can be called ``expander fractals". The property $\P$ that will be used is ``$X$ embeds with distortion $c/\e$ into Hilbert space", which is clearly a hereditary dilation-invariant metric property.

\begin{proof}[Proof of the second part of Theorem~\ref{thm:hausdorff into}]
Let $c,C, X_n$ be as above. Fix $\alpha>0$ and choose an integer $n\ge 2$ such that $n^{1/\alpha}>C\log n$. We may therefore use Corollary~\ref{cor:ub-first-lem} with $X=X_n$ and $\Psi=n^{1/\alpha}$. The resulting compact metric space $(Y,\rho)$ will then have Hausdorff dimension equal to $\alpha$. For every $\e\in (0,1)$ let $\P_\e$ be the property ``$X$ embeds with distortion $c/\e$ into Hilbert space". Then all the subsets $Z$ of $Y$ that embed into Hilbert space with distortion $c/\e$ satisfy $\dim_H(Z,\rho)\le \log_\Psi \left(n^{1-\e}\right)=(1-\e)\alpha=(1-\e)\dim_H(Y,\rho)$.
\end{proof}

\subsection{$G\left(n,1/2\right)$ fractals}\label{sec:random graph}
It is shown in~\cite{BLMN05} that there exists $K\in (1,\infty)$ such that for any $n\in \N$ there exists an $n$-point metric space $W_n$  such that for every $\delta\in (0,1)$ any subset of $W_n$ of size larger than $2\log_2 n +K\left(\delta^{-2}\log(2/\delta)\right)^2$ must incur distortion at least $2-\delta$ when embedded into Hilbert space. The space $W_n$ comes from a random construction: consider a random graph $G$ on $n$ vertices, drawn from the Erd\H{o}s-Reyni model $G(n,1/2)$ (thus every edge is present independently with probability $1/2$). The space $W_n$ is obtained from $G$ by declaring two vertices that are joined by an edge to be at distance $1$, and two distinct vertices that are not joined by an edge are declared to be at distance $2$. Therefore the positive distances in $W_n$ are either $1$ or $2$. This description of $W_n$ is implicit in~\cite{BLMN05} but follows immediately from the proof of~\cite{BLMN05}; see~\cite{BLMN05-low} for an alternative proof of this fact (yielding a worse asymptotic dependence on $\delta$ that is immaterial for our purposes). We will apply Theorem~\ref{lem:ub-second-lem} to $\{W_n\}_{n=2}^\infty$, thus obtaining compact metric spaces that can be called ``$G(n,1/2)$ fractals".

\begin{proof}[Proof of Theorem~\ref{thm:<2}] Let $K$ and $\{W_n\}_{n=2}^\infty$ be as in the above discussion.  Set $X_n=W_{n+\lceil 3^\alpha \rceil}$. Hence $|X_n|^{1/\alpha}> 2$.
Let $\P_n$ be the property ``$X$ embeds with distortion $2-1/\log \log n$ into Hilbert space". Then $\{P_n\}_{n\ge 20}$ is a  non-decreasing sequence of hereditary dilation-invariant metric properties. Moreover, $\bigcup_{n\ge 20} \P_n$ is the property ``$X$ embeds into Hilbert space with distortion smaller than $2$". By the above discussion, letting $m_n$ be the size of the largest subset of $X_n$ that has the property $\P_n$, we have
$ m_n \le 2\log_2n+O\left((\log\log)^2(\log\log\log n)^2\right)$. Hence $\limsup_{n \to \infty} \frac{\alpha \log m_n}{\log |X_n|}=0$.
By Theorem~\ref{lem:ub-second-lem} it follows there exists a compact metric space $(Y,\rho)$ with Hausdorff dimension $\alpha$
such that all of its subsets with positive Hausdorff dimension do not have property $\bigcup_{n\ge 20} \P_n$, namely any embeding of such a subset into Hilbert space must incur distortion at least 2.
\end{proof}

\bigskip

\noindent{\bf Acknowledgements.} We are grateful to Terence Tao for sharing with us his initial attempts to solve Question~\ref{Q:tao}. We thank Tam\'as Keleti, Andr\'as M\'ath\'e and Ond\v{r}ej Zindulka for helpful comments. We are also grateful to an anonymous referee who suggested a reorganization of our proof so as to improve the exposition.

\bibliographystyle{abbrv}
\bibliography{hausdorff}

\def\cprime{$'$} \def\cprime{$'$}
\begin{thebibliography}{10}

\bibitem{ACP05}
G.~Alberti, M.~Cs{\"o}rnyei, and D.~Preiss.
\newblock Structure of null sets in the plane and applications.
\newblock In {\em European {C}ongress of {M}athematics}, pages 3--22. Eur.
  Math. Soc., Z\"urich, 2005.

\bibitem{AK00}
L.~Ambrosio and B.~Kirchheim.
\newblock Rectifiable sets in metric and {B}anach spaces.
\newblock {\em Math. Ann.}, 318(3):527--555, 2000.

\bibitem{BBM}
Y.~Bartal, B.~Bollob{\'a}s, and M.~Mendel.
\newblock Ramsey-type theorems for metric spaces with applications to online
  problems.
\newblock {\em J. Comput. System Sci.}, 72(5):890--921, 2006.

\bibitem{BBM06}
Y.~Bartal, B.~Bollob{\'a}s, and M.~Mendel.
\newblock Ramsey-type theorems for metric spaces with applications to online
  problems.
\newblock {\em J. Comput. System Sci.}, 72(5):890--921, 2006.

\bibitem{BLMN05}
Y.~Bartal, N.~Linial, M.~Mendel, and A.~Naor.
\newblock On metric {R}amsey-type phenomena.
\newblock {\em Ann. of Math. (2)}, 162(2):643--709, 2005.

\bibitem{BLMN05-low}
Y.~Bartal, N.~Linial, M.~Mendel, and A.~Naor.
\newblock Some low distortion metric {R}amsey problems.
\newblock {\em Discrete Comput. Geom.}, 33(1):27--41, 2005.

\bibitem{BL}
Y.~Benyamini and J.~Lindenstrauss.
\newblock {\em Geometric nonlinear functional analysis. {V}ol. 1}, volume~48 of
  {\em American Mathematical Society Colloquium Publications}.
\newblock American Mathematical Society, Providence, RI, 2000.

\bibitem{BKRS00}
A.~Blum, H.~Karloff, Y.~Rabani, and M.~Saks.
\newblock A decomposition theorem for task systems and bounds for randomized
  server problems.
\newblock {\em SIAM J. Comput.}, 30(5):1624--1661 (electronic), 2000.

\bibitem{BFM86}
J.~Bourgain, T.~Figiel, and V.~Milman.
\newblock On {H}ilbertian subsets of finite metric spaces.
\newblock {\em Israel J. Math.}, 55(2):147--152, 1986.

\bibitem{Burago}
D.~Burago, Y.~Burago, and S.~Ivanov.
\newblock {\em A course in metric geometry}, volume~33 of {\em Graduate Studies
  in Mathematics}.
\newblock American Mathematical Society, Providence, RI, 2001.

\bibitem{Car67}
L.~Carleson.
\newblock {\em Selected problems on exceptional sets}.
\newblock Van Nostrand Mathematical Studies, No. 13. D. Van Nostrand Co., Inc.,
  Princeton, N.J.-Toronto, Ont.-London, 1967.

\bibitem{Chu89}
F.~R.~K. Chung.
\newblock Diameters and eigenvalues.
\newblock {\em J. Amer. Math. Soc.}, 2(2):187--196, 1989.

\bibitem{Dvo60}
A.~Dvoretzky.
\newblock Some results on convex bodies and {B}anach spaces.
\newblock In {\em Proc. {I}nternat. {S}ympos. {L}inear {S}paces ({J}erusalem,
  1960)}, pages 123--160. Jerusalem Academic Press, Jerusalem, 1961.

\bibitem{Fun11}
K.~Funano.
\newblock Two infinite versions of nonlinear {D}voretzky's theorem, 2011.
\newblock Preprint available at \url{http://arxiv.org/abs/1111.1627}.

\bibitem{Gro53-dvo}
A.~Grothendieck.
\newblock Sur certaines classes de suites dans les espaces de {B}anach et le
  th\'eor\`eme de {D}voretzky-{R}ogers.
\newblock {\em Bol. Soc. Mat. S\~ao Paulo}, 8:81--110 (1956), 1953.

\bibitem{How95}
J.~D. Howroyd.
\newblock On dimension and on the existence of sets of finite positive
  {H}ausdorff measure.
\newblock {\em Proc. London Math. Soc. (3)}, 70(3):581--604, 1995.

\bibitem{Hug04}
B.~Hughes.
\newblock Trees and ultrametric spaces: a categorical equivalence.
\newblock {\em Adv. Math.}, 189(1):148--191, 2004.

\bibitem{KKR94}
H.~Karloff, Y.~Rabani, and Y.~Ravid.
\newblock Lower bounds for randomized {$k$}-server and motion-planning
  algorithms.
\newblock {\em SIAM J. Comput.}, 23(2):293--312, 1994.

\bibitem{Kel95}
T.~Keleti.
\newblock A peculiar set in the plane constructed by {V}itu\v skin, {I}vanov
  and {M}elnikov.
\newblock {\em Real Anal. Exchange}, 20(1):291--312, 1994/95.

\bibitem{KMZ12}
T.~Keleti, A.~M\'ath\'e, and O.~Zindulka.
\newblock Hausdorff dimension of metric spaces and {L}ipschitz maps onto cubes.
\newblock Preprint, available at \url{http://arxiv.org/abs/1203.0686}, 2012.

\bibitem{LN05}
J.~R. Lee and A.~Naor.
\newblock Extending {L}ipschitz functions via random metric partitions.
\newblock {\em Invent. Math.}, 160(1):59--95, 2005.

\bibitem{Mattila}
P.~Mattila.
\newblock {\em Geometry of sets and measures in {E}uclidean spaces}, volume~44
  of {\em Cambridge Studies in Advanced Mathematics}.
\newblock Cambridge University Press, Cambridge, 1995.
\newblock Fractals and rectifiability.

\bibitem{MN07}
M.~Mendel and A.~Naor.
\newblock Ramsey partitions and proximity data structures.
\newblock {\em J. Eur. Math. Soc.}, 9(2):253--275, 2007.

\bibitem{MN11}
M.~Mendel and A.~Naor.
\newblock Ultrametric skeletons.
\newblock Preprint, available at~\url{http://arxiv.org/abs/1112.3416}, 2011.

\bibitem{MS99}
V.~Milman and G.~Schechtman.
\newblock An ``isomorphic'' version of {D}voretzky's theorem. {II}.
\newblock In {\em Convex geometric analysis ({B}erkeley, {CA}, 1996)},
  volume~34 of {\em Math. Sci. Res. Inst. Publ.}, pages 159--164. Cambridge
  Univ. Press, Cambridge, 1999.

\bibitem{Mil71}
V.~D. Milman.
\newblock A new proof of {A}. {D}voretzky's theorem on cross-sections of convex
  bodies.
\newblock {\em Funkcional. Anal. i Prilo\v zen.}, 5(4):28--37, 1971.

\bibitem{MP10}
P.~M{\"o}rters and Y.~Peres.
\newblock {\em Brownian motion}.
\newblock Cambridge Series in Statistical and Probabilistic Mathematics.
  Cambridge University Press, Cambridge, 2010.
\newblock With an appendix by Oded Schramm and Wendelin Werner.

\bibitem{NT-fragmentations}
A.~Naor and T.~Tao.
\newblock Scale-oblivious metric fragmentation and the nonlinear {D}voretzky
  theorem, 2010.
\newblock Preprint available at \url{http://arxiv.org/abs/1003.4013}. To appear
  in Israel J. Math.

\bibitem{Sag94}
H.~Sagan.
\newblock {\em Space-filling curves}.
\newblock Universitext. Springer-Verlag, New York, 1994.

\bibitem{Sch06}
G.~Schechtman.
\newblock Two observations regarding embedding subsets of {E}uclidean spaces in
  normed spaces.
\newblock {\em Adv. Math.}, 200(1):125--135, 2006.

\bibitem{SVY09}
C.~Sommer, E.~Verbin, and W.~Yu.
\newblock Distance oracles for sparse graphs.
\newblock In {\em 2009 50th {A}nnual {IEEE} {S}ymposium on {F}oundations of
  {C}omputer {S}cience ({FOCS} 2009)}, pages 703--712. IEEE Computer Soc., Los
  Alamitos, CA, 2009.

\bibitem{Tal87}
M.~Talagrand.
\newblock Regularity of {G}aussian processes.
\newblock {\em Acta Math.}, 159(1-2):99--149, 1987.

\bibitem{Tal05}
M.~Talagrand.
\newblock {\em The generic chaining}.
\newblock Springer Monographs in Mathematics. Springer-Verlag, Berlin, 2005.
\newblock Upper and lower bounds of stochastic processes.

\bibitem{Tal11}
M.~Talagrand.
\newblock {\em Upper and Lower Bounds for Stochastic Processes}.
\newblock 2011.
\newblock Modern Methods and Classical Problems. Forthcoming book.

\bibitem{TZ05}
M.~Thorup and U.~Zwick.
\newblock Approximate distance oracles.
\newblock {\em J. ACM}, 52(1):1--24 (electronic), 2005.

\bibitem{Urb09}
M.~Urba{\'n}ski.
\newblock Transfinite {H}ausdorff dimension.
\newblock {\em Topology Appl.}, 156(17):2762--2771, 2009.

\bibitem{VT79}
I.~A. Vestfrid and A.~F. Timan.
\newblock A universality property of {H}ilbert spaces.
\newblock {\em Dokl. Akad. Nauk SSSR}, 246(3):528--530, 1979.

\bibitem{VIM63}
A.~G. Vitu{\v{s}}kin, L.~D. Ivanov, and M.~S. Mel{\cprime}nikov.
\newblock Incommensurability of the minimal linear measure with the length of a
  set.
\newblock {\em Dokl. Akad. Nauk SSSR}, 151:1256--1259, 1963.

\bibitem{Wul12}
C.~Wulff-{N}ilsen.
\newblock Approximate distance oracles with improved query time.
\newblock Preprint, available at~\url{http://arxiv.org/abs/1202.2336}, 2011.

\end{thebibliography}

\end{document}